\documentclass[12pt]{article}

\usepackage{amsfonts,amsmath,amssymb,latexsym,xcolor,mathrsfs,tikz,breqn,extarrows}
\usepackage{amsthm,authblk}
\usepackage{placeins}

\usepackage{amssymb}
\usepackage{amsmath}
\usepackage{amsthm}
\usepackage{tikz}
\usepackage{mathtools}
\usepackage{calc}
\usepackage{tikz}
\usepackage{pgfplots}
\usetikzlibrary{patterns,arrows,decorations.pathreplacing}
\usepackage{diagbox}
\usepackage{makecell}

\def\q{\hfill\rule{1ex}{1ex}}
\def\0{\emptyset}

\def\q{\hfill\rule{1ex}{1ex}}

\usepackage{amsthm}

\newtheorem{theorem}{Theorem}[section]
\newtheorem{definition}[theorem]{Definition}
\newtheorem{lemma}[theorem]{Lemma}

\newtheorem{observation}[theorem]{Observation}
\newtheorem{cor}[theorem]{Corollary}
\newtheorem{prop}[theorem]{Proposition}
\theoremstyle{definition}
\newtheorem{remark}[theorem]{Remark}
\theoremstyle{plain}

\newtheorem{prob}[theorem]{Problem}

\renewcommand\theenumi{\rm \roman{enumi}}
\renewcommand\labelenumi{\rm (\theenumi)}

\usepackage{algorithm}
\usepackage{algpseudocode}
\usepackage{graphicx}
\usepackage{amsmath}
\usepackage{pstricks,pgf,subcaption,amssymb}
\begin{document}
\title{\bf On rainbow saturated graphs with minimum number of edges
	 }
\author[]{
Yanzhe Qiu} 
\author[]{
Mei Lu} 
\author[]{
Yilin Pan} 
\author[]{
Yiduo Xu\thanks{Corresponding author. E-mail:\texttt{xyd23@mails.tsinghua.edu.cn}}\;}

\affil[]{\small Department of Mathematical Sciences, Tsinghua University, Beijing 100084, China}
\date{}

\maketitle\baselineskip 16.3pt

\begin{abstract}
	Let $F$ be a fixed graph without isolated vertices. An edge-colored graph is $F$-rainbow saturated if it contains no rainbow copy of $F$, but the addition of any missing edge in any color creates a rainbow copy of $F$. The rainbow saturation number $rsat(n,F)$ is the minimum number of edges in such a graph on $n$ vertices. We prove a dichotomy governed by isolated edges: if $F$ contains an isolated edge, then $rsat(n,F)=O(1)$ for all sufficiently large $n$, while if $F$ has no isolated edge, then $rsat(n,F)=\Theta(n)$. The linear lower bound is expressed in terms of a directed weight parameter $\eta(F)$ and establishes the linear half of the dichotomy; in several cases it also strengthens the Cameron--Puleo type coefficient. For the bounded half, we construct rainbow saturated graphs for targets of the form $H\cup K_2$.
	As an application of these constructions, we determine the asymptotically tight behavior for the rainbow saturation number of the generalized friendship graph $F_{t,p,q}=tK_p\vee K_q$, proving that
	\[
		rsat(n,F_{t,p,q})=(p+q-1)n+O(1)
	\]
	for fixed $t\geq 2$, $p\geq 2$ and $q\geq 1$ as $n\to\infty$.

	\end {abstract}
	
	edge-coloring

\section{Introduction}

All graphs considered are finite, simple and undirected, and we use standard graph-theoretic notation. In particular, $|G|$ and $e(G)$ denote the order and size of a graph $G$. Unions of graphs are understood to be vertex-disjoint, $tG$ denotes the union of $t$ copies of $G$, and $G_1\vee G_2$ denotes the join of $G_1$ and $G_2$. Let $G=(V,E)$ be a graph, where $E(G)$ is the edge set and  $V(G)$ is the vertex  set. For $S\subseteq V(G)$ and $u\in V(G)$, denote $N_S(u)=\{v\in S: uv\in E(G)\}$ and $d_S(u)=|N_S(u)|$. If $S=V(G)$, then we write
 $N_G(u)=N_{V(G)}(u)$, $d_G(u)=d_{V(G)}(u)$ and call $d_G(u)$ the degree of $u$. The minimum   degree of $G$ is  $\delta(G)=\min\{d_G(u):u\in V(G)\}$.  When there is no ambiguity, we ignore the subscript $G$ in above notation. If $S,C$ are two sets of $V(G)$, we denote $E(S,C)=\{uv\in E(G):u\in S,v\in C\}$. Let $[n]=\{1,\ldots,n\}$, where $n$ is a positive integer.

An \emph{edge-colored graph} is a pair $(G,\mathcal C)$, where $\mathcal C:E(G)\to\mathbb N$ is an arbitrary map; in particular, the coloring need not be proper. Denote $X_{\mathcal C}=\{\mathcal C(e):e\in E(G)\}$. A subgraph $H\subseteq G$ is \emph{rainbow} if the restriction of $\mathcal C$ to $E(H)$ is injective.
Whenever the coloring is specified or clear from the context, we suppress it from the notation and use $G$ to denote the edge-colored graph $(G,\mathcal C)$. 
When several edge-colored graphs are combined, their color sets are relabeled to be pairwise disjoint unless an explicit color identification is stated.
An edge set $E'\subseteq E(G)$ \emph{avoids} a color $\alpha$ if no edge in $E'$ has the color $\alpha$. Accordingly, a subgraph $H\subseteq G$ avoids a color $\alpha$ if $E(H)$ avoids the color $\alpha$, and we say it avoids a vertex $v$ if $v\notin V(H)$.

We say that $(G,\mathcal C)$ is \emph{$F$-rainbow free} if it contains no rainbow copy of $F$.
For a nonedge $e \notin E(G)$ and a color $\alpha\in\mathbb N$, write $(G,\mathcal C)+_\alpha e$ for the edge-colored graph obtained by adding $e$ to $G$ and coloring $e$ by the color $\alpha$.
When the coloring is clear on $G$, we write simply $G+_\alpha e$.
We say that $(G,\mathcal C)$ is \emph{$F$-rainbow-semisaturated} if $(G,\mathcal C)+_\alpha e$ contains a rainbow copy of $F$ containing $e$ for every nonedge $e$ and every $\alpha\in\mathbb N$.
It is \emph{$F$-rainbow saturated} if it is both $F$-rainbow free and $F$-rainbow-semisaturated.
The rainbow semisaturation and saturation numbers of $F$ are, respectively,
\[
\begin{aligned}
	rssat(n,F)&=\min\{e(G): |G|=n \text{ and $(G,\mathcal C)$ is $F$-rainbow-semisaturated for some $\mathcal C$}\},\\
	rsat(n,F)&=\min\{e(G): |G|=n \text{ and $(G,\mathcal C)$ is $F$-rainbow saturated for some $\mathcal C$}\}.
\end{aligned}
\]
Clearly $rsat(n,F)\geq rssat(n,F)$.

The saturation problem was initiated by Erd\H os, Hajnal and Moon~\cite{EHM}; we refer to~\cite{survey} for a survey.
A basic structural question is when the saturation number is degenerate.
K\'aszonyi and Tuza~\cite{KT} proved that an isolated edge is exactly the obstruction to linear growth: if $F$ contains an isolated edge, then $sat(n,F)$ is bounded by an absolute constant $c=c(F)$; otherwise $sat(n,F)=\Theta(n)$. Cameron and Puleo~\cite{Ca} later gave a stronger lower bound in terms of a local edge-weight parameter, with further refinements obtained by Buchanan and Rombach~\cite{BR}.

	The rainbow saturation number was introduced by Gir\~ao, Lewis and Popielarz~\cite{Gir}. Behague, Johnston, Letzter, Morrison and Ogden~\cite{BEH} proved that $rsat(n,F)=O(n)$ for every fixed graph $F$, and conjectured that there exists a constant $c=c(F)$ such that $rsat(n,F)=(c+o(1))n$ as $n\to\infty$. Chakraborti, Hendrey, Lund and Tompkins~\cite{CHA} proved that $\lim_{n\to\infty}rsat(n,K_r)/n$ exists and gave bounds on this limit. Xu, He and Lu~\cite{Xu} determined $rsat(n,C_4)=3\lceil(n-1)/2\rceil$ for $n\geq5$ and gave bounds for $rsat(n,C_r)$ when $r\geq5$.

	Throughout the paper, we only consider graphs  with no isolated vertices.
	A basic open problem is to determine which graphs have bounded rainbow saturation number. Our main theorem answers this problem.
\vspace{0.3em}	
		
\begin{theorem}\label{T11}
Let $F$ be a fixed graph. For all sufficiently large $n$, we have $rsat(n,F)=O(1)$
if $F$ contains an isolated edge; otherwise $rsat(n,F)=\Theta(n)$.
\end{theorem}
\vspace{0.3em}	

Theorem~\ref{T11} is an immediate consequence of the next two theorems.
The lower bound is obtained from the directed-weight estimate in the following theorem.
The parameter $\eta(F)\geq 1$ is defined in Section~2.
\vspace{0.3em}

\begin{theorem}\label{TLower}
Let $F$ be a graph with no isolated edge. Then for every $n \geq |F|$,
\[
	rssat(n,F) \geq \frac{\eta(F)}{2}n-\frac{\eta(F)^2-2\eta(F)+2}{2}.
\]
In particular, $rsat(n,F) \geq \frac{\eta(F)}{2}n-O(1)$.
\end{theorem}
\vspace{0.1em}

The upper bound is supplied by a bounded construction for every graph containing an isolated edge.
\vspace{0.2em}

\begin{theorem}\label{TIsolated}
Let $F$ be a fixed graph containing an isolated edge. Then $rsat(n,F)=O(1)$ for all sufficiently large $n$.
\end{theorem}
\vspace{0.1em}

Let $F_{t,p,q}= tK_p \vee K_q$ be the generalized friendship graph. We call $K_q$ and $K_p $ the {\em center} and the {\em petal} of $F_{t,p,q}$, respectively. Note that $F_{t,p,q}$ has $t$ petals.
The bounded gadgets developed for Theorem~\ref{TIsolated} can also be used to supply the edges of petals in this connected graph. We obtain the following asymptotic formula.
	
\vspace{0.3em}

\begin{theorem}\label{TFriend}
	For fixed integers $t \geq 2$, $p \geq 2$ and $q \geq 1$, $rsat(n,F_{t,p,q})=(p+q-1)n+O(1)$ for sufficiently large $n$.
\end{theorem}
\vspace{0.1em}

Note that if $t\ge 2$, then $F_{t,1,1}=K_{1,t}$ and $rsat(n,F_{t,1,1})=\frac{t-1}{2}n+O(1)$ for odd $t\ge 3$ given in Remark \ref{R24}. If $t=1$, then $F_{1,p,q}=K_{p+q}$ and the upper bound of $rsat(n,F_{1,p,q})$ is given in \cite{BEH}.

	The rest of this paper is organized as follows.
	Section~2 introduces the directed weight parameters and the basic color-avoidance criterion, and then proves the linear lower bound for graphs without isolated edges.
	Section~3 develops bounded constructions for graphs containing an isolated edge and completes the proof of Theorem~\ref{T11}.
	Section~4 applies these constructions to generalized friendship graphs, and Section~5 concludes with several remarks and open problems.

\section{Linear lower bounds}

For  an edge-colored graph $(G,\mathcal C)$, fix $e\notin E(G)$ and denote
\[\mathcal F_e=\{J: J \mbox{~is a rainbow copy  of~} F \mbox{~in~} (G,\mathcal C)+_\beta e \mbox{~with~} e\in E(J) \mbox{ for some }\beta \notin X_{\mathcal C}\}.\]

\vspace{0.3em}
	
\begin{prop}\label{P21}
	An edge-colored graph $(G,\mathcal C)$ is $F$-rainbow-semisaturated if and only if, for every nonedge $e$ of $G$ and every color $\alpha\in\mathbb N$, there exists $J\in\mathcal F_e$ such that no edge of $J-e$ has the color $\alpha$. If $(G,\mathcal C)$ is also $F$-rainbow free, this condition is equivalent to $(G,\mathcal C)$ being $F$-rainbow saturated.
\end{prop}
\begin{proof} The necessity is obvious by the definition of $F$-rainbow-semisaturated. For the sufficiency, let $e$ be a  nonedge of $G$ and $\alpha\in\mathbb N$. Since there exists $J\in\mathcal F_e$ such that no edge of $J-e$ has the color $\alpha$, $(G,\mathcal C)+_\alpha e$ contains a rainbow copy of $F$ containing $e$.
The final assertion follows from the definition of rainbow saturation.
\end{proof}

For an ordered edge $uv\in E(F)$, put
\[
	\omega(u,v)=2|N_F(u)\cap N_F(v)|+|N_F(v)\setminus N_F(u)|.
\]
Define
\[
	\omega^*(F)=\min_{uv\in E(F)}\max\{\omega(u,v),\omega(v,u)\},\qquad
	\omega_*(F)=\min_{uv\in E(F)}\min\{\omega(u,v),\omega(v,u)\}.
\]
The first quantity, $\omega^*(F)$, is the Cameron--Puleo weight parameter \cite{Ca}. For a graph $F$ without isolated edges, define
\[
	\eta(F)=\max\{\omega^*(F)-1,\omega_*(F)\}.
\]
Then $\omega^*(F)-1\leq\eta(F)\leq\omega^*(F)$. Also it is clear that $\eta(F) \geq 1$.

\begin{lemma}\label{L41}
Let $F$ be a graph without isolated edge. If $(G,\mathcal C)$ is an $F$-rainbow-semisaturated graph on $n \geq |F|$ vertices, then
\[
	e(G) \geq \frac{\eta(F)}{2}n-\frac{\eta(F)^2-2\eta(F)+2}{2}.
\]
\end{lemma}
\noindent {\bf Proof.}
Let $x^*\in V(G)$ with $d(x^*)=\delta(G)$, let $B=N(x^*)$ and $\overline{B}=V(G) \setminus B$. We first derive two local estimates for vertices in $\overline{B}\setminus\{x^*\}$.

Fix $y \in \overline{B} \setminus \{x^*\}$, and choose a color $\beta\notin X_{\mathcal C}$. Since $(G,\mathcal C)$ is $F$-rainbow-semisaturated, $(G,\mathcal C)+_\beta x^*y$ contains a rainbow copy of $F$ using $x^*y$. Let $\varphi$ be the corresponding embedding of $F$, with an edge $uv\in E(F)$ mapped to $x^*y$. Assume that $d_F(u) \leq d_F(v)$. Write $a=|N_F(v)\setminus N_F(u)|$ and $b=|N_F(u) \cap N_F(v)|$. Then $a\ge 1$ and $\omega(u,v)=a+2b \geq \omega^*(F)$.

We first obtain upper bound of  $\omega^*(F)$. If $y=\varphi(u)$, then $d_G(y)\geq d_G(x^*) \geq d_F(v)-1=a+b-1$; if $y=\varphi(v)$, then $d_G(y) \geq d_F(v)-1=a+b-1$ as well. Moreover, $\varphi(N_F(u)\cap N_F(v))\subseteq N_G(x^*)=B$, so $y$ has at least $b$ neighbors in $B$. Hence
\[
	2d_B(y)+d_{\overline{B}}(y)=d_G(y)+d_B(y)\geq 2b+a-1=\omega(u,v)-1\geq \omega^*(F)-1.
\]

Since $F$ has no isolated edge, $y$ is not isolated in $G$.
Choose a color $i$ appearing on an edge incident with $y$. By Proposition~\ref{P21}, $(G,\mathcal C)+_i x^*y$ contains a rainbow copy of $F$ using $x^*y$ such that all its edges in $G$ avoid $i$.
Let $\psi$ be the corresponding embedding, with $x^*=\psi(p)$ and $y=\psi(q)$ for an edge $pq\in E(F)$. Put $a'=|N_F(q)\setminus N_F(p)|$ and $b'=|N_F(p)\cap N_F(q)|$. The copy gives $b'$ edges from $y$ to $B$ and $a'-1$ other edges incident with $y$, all avoiding color $i$.
The chosen edge of color $i$ incident with $y$ is not used in this copy, so it contributes at least one additional unit to $2d_B(y)+d_{\overline{B}}(y)$. Thus
\[
	2d_B(y)+d_{\overline{B}}(y)\geq 2b'+a'=\omega(p,q)\geq \omega_*(F).
\]
Combining the two  bounds, for every $y \in \overline{B} \setminus \{x^*\}$, we have
\[
	2d_B(y)+d_{\overline{B}}(y)\geq \eta(F).
\]

We now obtain the desired lower bound. If $d_G(x^*)\geq \eta(F)$, then $e(G)\geq \frac{\eta(F)}{2}n$. If $d_G(x^*)\leq \eta(F)-1$, then using $\sum_{x \in B} d_G(x) \geq \sum_{y \in \overline{B}} d_B(y)$, we have
\begin{align*}
	e(G) &= \frac{1}{2}\left(\sum_{x \in B} d(x) + \sum_{y \in \overline{B}} d(y)\right)
	\geq \frac{1}{2}\sum_{y \in \overline{B}} \bigl(2d_B(y)+d_{\overline{B}}(y)\bigr) \\
	&\geq \frac{1}{2}\left(2d_G(x^*)+\eta(F)(|\overline{B}|-1)\right)
	= \frac{\eta(F)}{2}n-\frac{(\eta(F)-2)d_G(x^*)+\eta(F)}{2} \\
	&\geq \frac{\eta(F)}{2}n-\frac{\eta(F)^2-2\eta(F)+2}{2}.
\end{align*} \qed
\vspace{0.3em}

Now we are going to prove Theorem~\ref{TLower}.
\vspace{0.3em}

\medskip\noindent {\bf Proof of Theorem~\ref{TLower}.}
Lemma~\ref{L41} gives the stated lower bound for $rssat(n,F)$. Since $rsat(n,F) \geq rssat(n,F)$, the final assertion follows. \qed
\vspace{0.5em}

The following corollary is obvious by Theorem~\ref{TLower}.
\begin{cor}\label{C23}
Let $F$ be a graph without isolated edge. If $\omega_*(F)=\omega^*(F)$, then
\[
	rsat(n,F)\geq \frac{\omega^*(F)}{2}n-O(1).
\]
In particular, if $d_F(u)=d_F(v)$ for every edge $uv \in E(F)$, then $\eta(F)=\omega^*(F)$. Hence, if $F$ is $d$-regular and $m=\min_{uv\in E(F)}|N_F(u)\cap N_F(v)|$,
then
\[
	rsat(n,F)\geq \frac{d+m}{2}n-O(1).
\]
\end{cor}
\vspace{0.3em}

\begin{remark}\label{R24}
The coefficient in Theorem~\ref{TLower} is sharp for specific graphs. Let $F=K_{1,r}$ with odd $r\geq 3$. Then $\eta(F)=r-1$. For every $n \geq r+1$, an $(r-1)$-regular graph $G$ on $n$ vertices exists, and we give $G$ a rainbow coloring. Since $\Delta(G)=r-1$,  $G$ contains no rainbow $K_{1,r}$.
For every nonedge $xy$ and every color $c$, at least one endpoint is incident with no edge of color $c$ in $G$, say $x$.
Then $(G+_c xy)[N_G(x)\cup \{x,y\}]$ contains a rainbow copy of $K_{1,r}$.
Therefore
\[
	rsat(n,K_{1,r})\leq \frac{r-1}{2}n,
\]
while Theorem~\ref{TLower} gives $rsat(n,K_{1,r})\geq \frac{r-1}{2}n-O(1)$.
\end{remark}
\vspace{0.0em}

\begin{remark}\label{R25}
A threshold graph is a graph obtained from a single vertex by repeatedly adding either an isolated vertex or a vertex adjacent to all existing vertices (for example, stars and cliques).
Cameron and Puleo~\cite{Ca} proved that every threshold graph is sharp for the saturation lower bound
\[
	sat(n,F)\geq \frac{\omega^*(F)-1}{2}n-O(1).
\]
It is natural to ask whether a similar property holds for rainbow saturation. But the answer is negative.
For $K_s$ we have $\eta(K_s)=\omega_*(K_s)=2s-3$, and  Theorem~\ref{TLower} gives $rsat(n,K_3)\geq \frac{3}{2}n-O(1)$, whereas the exact value $rsat(n,K_3)=2n-4$ for sufficiently large $n$ is known in~\cite{CHA}. Thus threshold graphs give useful examples, but they do not form a sharp family for Theorem~\ref{TLower}.
\end{remark}
\vspace{0.3em}

\section{Constructions for graphs with an isolated edge}

We now turn to the upper-bound half of the dichotomy.

\begin{prop}\label{PNonClique}
Let $H$ be a graph that is not complete and has no isolated vertices. Then for every $n\geq |H|+3$,
\[
	rsat(n,H \cup K_2)\leq \binom{|H|+1}{2}.
\]
\end{prop}
\begin{proof}
Let $h=|H|$ and $G = K_{h+1} \cup (n-h-1)K_1$. Then $e(G)=\binom{h+1}{2}$. Given a rainbow coloring of $G$, we shall show that $G$ is $(H \cup K_2)$-rainbow saturated.
Since $H$ has no isolated vertices, it is clear that $G$ is $(H \cup K_2)$-rainbow free.
	
Partition $V(G)=M \cup I$ where $M$ is the set of $h+1$ vertices in $K_{h+1}$ and $I$ is the set of all isolated vertices.
Fix a nonedge $e=xy$ and an arbitrary color $c$, where $x \in I $. Since $H$ is not complete, for any vertex $v \in M$, there must exist a rainbow copy of $H$ in $G[M \setminus \{v\}]$ avoiding color $c$. Thus $G+_c xy$ contains a rainbow copy of $H \cup K_2$ whether $y \in M$ or not.
Therefore $G$ is $(H \cup K_2)$-rainbow saturated, and hence
\[
rsat(n,H \cup K_2)\leq \binom{h+1}{2}.
\]
\end{proof}

Proposition~\ref{PNonClique} gives the desired upper bound for all $H \cup K_2$ where $H$ has no isolated vertices and is not complete. But the construction fails when $H$ is complete, so we need to develop another construction for $K_s \cup K_2$.

\begin{definition}\label{D510}
	Let $H$ be a graph. An edge-colored graph $(G,\mathcal C)$ is \textit{$\mathcal{R}(H,2)$-saturated} if it satisfies
	{\renewcommand\theenumi{\alph{enumi}}\renewcommand\labelenumi{(\theenumi)}
	\begin{enumerate}
		\item $G$ contains no rainbow copy of $H \cup K_2$;
		\item for every vertex $v \in V(G)$ and every color $c$, $G$ contains a rainbow copy of $H$ avoiding both $v$ and  $c$;
		\item for every nonedge $\tilde e \notin E(G)$ and every color $c$,  $G+_c\tilde e$ has a rainbow copy of $H\cup K_2$ containing $\tilde e$.
	\end{enumerate}}
	If only (b) and (c) hold, $(G,\mathcal C)$ is \textit{$\mathcal{R}(H,2)$-semisaturated}. When the coloring is fixed and clear from context, we suppress $\mathcal C$ and refer simply to $G$.
\end{definition}

\begin{definition}\label{D511}
	A graph $G$ is \textit{endpoint-less} if it has no isolated vertices and the set $M=\{v\in V(G):d_G(v)=1\}$ induces a matching covering $M$.
\end{definition}

\begin{observation}\label{O512}
	If $G_1$ and $G_2$ are endpoint-less, so is $G_1 \cup G_2$.
\end{observation}

\begin{prop}\label{P513}
	Let $H$ be an endpoint-less graph. Then $(G,\mathcal C)$ is $\mathcal{R}(H,2)$-saturated if and only if $G\cup kK_1$ is $(H \cup K_2)$-rainbow saturated for every $k \geq 1$.
\end{prop}
\begin{proof}
Let $X$ be the set of the $k$ isolated vertices in $G\cup kK_1$. Since $H$ has no isolated vertices, every copy of $H$ in $G\cup kK_1$ is contained in $G$. Hence the property~(a) is equivalent to $G\cup kK_1$ containing no rainbow copy of $H\cup K_2$.

Fix a nonedge $e \notin E(G\cup kK_1)$ and a color $c$. If both endpoints of $e$ lie in $G$, the required condition for $(G\cup kK_1)+_c e$ is exactly the property~(c).

Suppose next that $e=vx$, where $v\in V(G)$ and $x\in X$. By the property~(b), $G$ contains a rainbow copy of $H$ avoiding both $v$ and $c$; together with $vx$, it forms a rainbow copy of $H\cup K_2$ in $(G\cup kK_1)+_c vx$.
Conversely, suppose that $(G\cup kK_1)+_c vx$ contains a rainbow copy of $H\cup K_2$ containing $vx$.
Since $x$ has degree one in $(G\cup kK_1)+_c vx$ and $H$ is endpoint-less, $vx$ corresponds to a $K_2$-component of the target graph.
If it corresponds to a $K_2$-component of $H$, interchange this component with the extra $K_2$-component.
In either case, the remaining copy of $H$ lies in $G$ and avoids both $v$ and $c$, which is the property~(b).

Finally, if $e=xy$ with $x,y\in X$, apply the property~(b) with an arbitrary vertex of $G$ to obtain a rainbow copy of $H$ avoiding $c$; together with $xy$, it forms a rainbow copy of $H\cup K_2$ in $(G\cup kK_1)+_cxy$. Thus the properties~(a)--(c) are together equivalent to $G\cup kK_1$ being $(H\cup K_2)$-rainbow saturated.
\end{proof}
\vspace{0.3em}

\begin{remark}\label{R514} By  Proposition \ref{P513},
	for an endpoint-less graph $H$, the problem of determining $rsat(n,H \cup K_2)$ for $n$ sufficiently large is equivalent to finding an $\mathcal{R}(H,2)$-saturated graph with the fewest edges.
\end{remark}

\begin{lemma}\label{L515}
	Let $H_1,H_2$ be endpoint-less graphs and let $G_1,G_2$ be $\mathcal{R}(H_1,2)$- and $\mathcal{R}(H_2,2)$-semisaturated, respectively. If $G_1$ and $G_2$ use disjoint color sets, then $G= G_1 \cup G_2$ is $\mathcal{R}(H_1 \cup H_2,2)$-semisaturated.
	
	Moreover, if both $G_1$ and $G_2$ are saturated and $G_2$ contains no rainbow copy isomorphic to a  component of $H_1$, then $G_1 \cup G_2$ is $\mathcal{R}(H_1 \cup H_2,2)$-saturated.
\end{lemma}
\noindent {\bf Proof.}
We first show that $G$ satisfies the properties~(b) and~(c) of Definition~\ref{D510}.
For (b), given a vertex $v\in V(G)$, say $v\in V(G_1)$, and a color $c$. Since $G_1,G_2$ are $\mathcal{R}(H_1,2)$- and $\mathcal{R}(H_2,2)$-semisaturated respectively, the property~(b) gives a rainbow copy $J_1$ of $H_1$ in $G_1$ avoiding $v$ and $c$, and a rainbow copy $J_2$ of $H_2$ in $G_2$ avoiding $c$. Then $J_1\cup J_2$ is the required copy. For (c), let $v_1v_2$ be a nonedge in $G$ and $c$ be a  color. If $v_1\in V(G_1)$ and $v_2\in V(G_2)$, by $G_1,G_2$ being $\mathcal{R}(H_1,2)$- and $\mathcal{R}(H_2,2)$-semisaturated,  the property~(b) gives a rainbow copy of $H_i$ in $G_i$ avoiding $v_i,c$ for $i=1,2$; together with $v_1v_2$ they form a rainbow copy of $H_1\cup H_2\cup K_2$. Assume $v_1,v_2\in V(G_1)$ or $v_1,v_2\in V(G_2)$, say $v_1,v_2\in V(G_1)$. Since $G_1,G_2$ are $\mathcal{R}(H_1,2)$- and $\mathcal{R}(H_2,2)$-semisaturated respectively, by using the property (c) on $G_1$ and the property (b) on $G_2$, there is a rainbow copy of $H_1\cup K_2$ in $G_1+_cv_1v_2$ and a rainbow copy of $H_2$ in $G_2$ avoiding $c$. Then $G+_cv_1v_2$  contains a rainbow copy of $H_1\cup H_2\cup K_2$ containing $v_1v_2$.

For the final assertion, we just need to show that $G$ contains no rainbow copy of $H_1\cup H_2 \cup K_2$.
Note that $G_2$ contains no rainbow copy isomorphic to a component of $H_1$. Assume for a contradiction that $G=G_1\cup G_2$ contains a rainbow copy $A$ of $H_1\cup H_2\cup K_2$. Since every component of $H_1$ in $A$ must lie in $G_1$, there is a rainbow copy $B$ of $H_1$ in $G_1$. Suppose there is  a component of $H_2\cup K_2$ in $A$ that is in $G_1$. Since $H_2$ has no isolated vertices, that component would contain an edge $u_1u_2$ disjoint from $B$. Then  $B\cup G[\{u_1,u_2\}]$ would  be a rainbow copy of $H_1\cup K_2$ in $G_1$, contradicting the saturation of $G_1$. Consequently, the copy of $H_2\cup K_2$ in $A$ must lie in $G_2$, contradicting the saturation of $G_2$. \qed
\vspace{0.3em}


\begin{definition}\label{D516}
	For $m \geq 3$, we construct an  edge-colored graph $A_m$ of order $2m$ and size $4\binom{m}{2}$. Let $V(A_m)=\{v_{j,k}:j \in \{1,2\}, 1 \leq k \leq m\}$ and $ v_{j_1,k_1}v_{j_2,k_2}\in E(A_m)$ if and only if $k_1\not=k_2$ for any $j_1,j_2 \in \{1,2\}$ and $1 \leq k_1,k_2 \leq m$.
We give a map $\mathcal{C}:~E(A_m)\rightarrow \{a_{j,k}, b_{j,k}:1 \leq j < k \leq m\}$ as follows: for any $v_{j_1,k_1}v_{j_2,k_2}\in E(A_m)$,
$$\mathcal{C}(v_{j_1,k_1}v_{j_2,k_2})=\left\{
\begin{array}{ll}
a_{\min\{k_1,k_2\},\max\{k_1,k_2\}} & \mbox{if $j_1=j_2$},\\
b_{\min\{k_1,k_2\},\max\{k_1,k_2\}} & \mbox{if $j_1\not=j_2$}.
\end{array}
\right.
$$
	We also define $A_2$ to be an edge-colored copy of $K_4$ using two colors, where $\mathcal{C}(v_1v_2)=\mathcal{C}(v_1v_3)=\mathcal{C}(v_2v_4)=\mathcal{C}(v_3v_4)$ and $\mathcal{C}(v_2v_3)=\mathcal{C}(v_1v_4)$ (see Figure 1).
\end{definition}

Note that when $m\ge 3$, $X_{\mathcal{C}}=\{a_{j,k}, b_{j,k}:1 \leq j < k \leq m\}$ and then $|X_{\mathcal{C}}|=2\binom{m}{2}$ and each color is  used exactly twice in the map $\mathcal{C}$.  Denote $V_i=\{v_{1,i},v_{2,i}\}$ for $1\le i\le m$. Then $V(A_m)=V_1\cup\cdots\cup V_m$ and each $V_i$ is an independent set. For any $K_m$ in $A_m$, $|V(K_m)\cap V_i|=1$ for $1\le i\le m$. By the definition of the map $\mathcal{C}$, $A_m[\{v_{1,k}:1 \leq  k \leq m\}]$ and $A_m[\{v_{2,k}:1 \leq  k \leq m\}]$ are rainbow copies of $K_m$ with the same color set $\{a_{j,k}:1 \leq j < k \leq m\}$. Also $A_m$ is $K_{m+1}$-free when $m\ge 3$.
Then we have the following result.

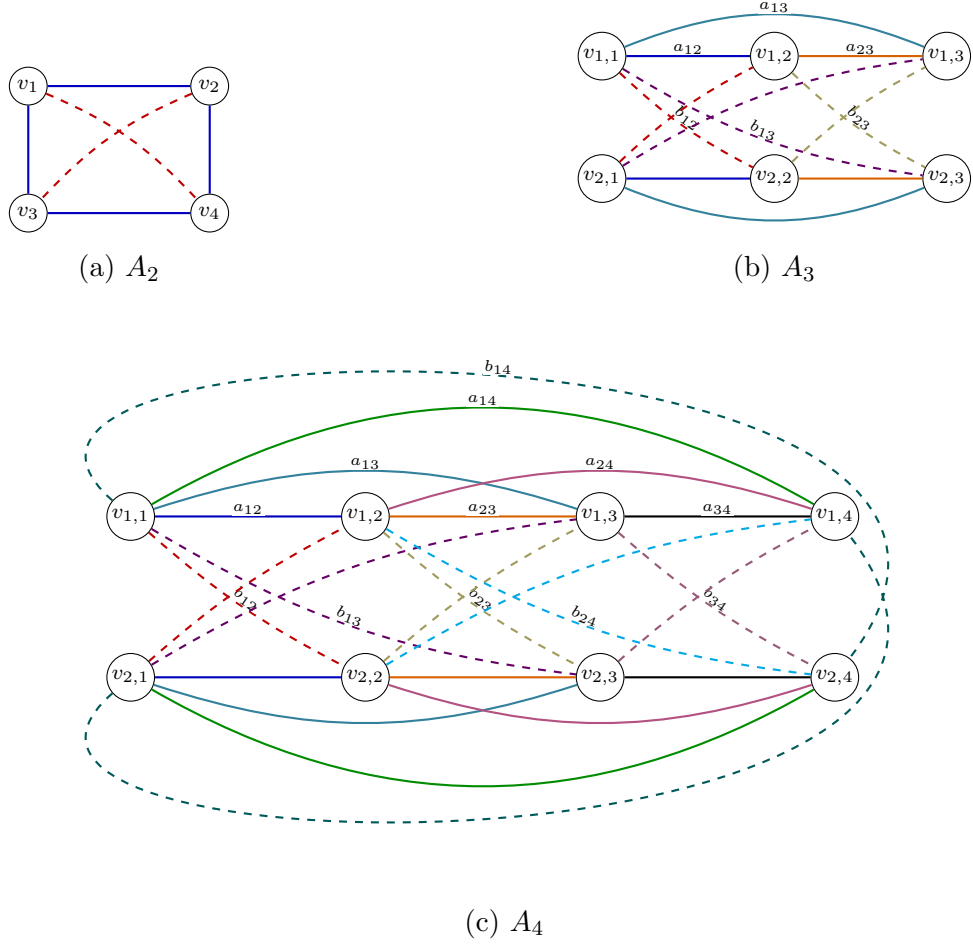
\begin{figure}[!ht]
\centering
\tikzset{
	Avertex/.style={circle,draw,fill=white,inner sep=1pt,minimum size=5mm,font=\scriptsize},
	Aedge/.style={line width=0.55pt,thick},
	Bedge/.style={line width=0.55pt,dashed,thick},
	Alabel/.style={font=\tiny,fill=white,inner sep=0.4pt},
	Acol12/.style={draw=blue!75!black},
	Acol13/.style={draw=cyan!55!black},
	Acol14/.style={draw=green!55!black},
	Acol23/.style={draw=red!55!yellow!85!black},
	Acol24/.style={draw=magenta!70!black},
	Acol34/.style={draw=black},
	Bcol12/.style={draw=red!75!black},
	Bcol13/.style={draw=violet!80!black},
	Bcol14/.style={draw=teal!70!black},
	Bcol23/.style={draw=yellow!55!black},
	Bcol24/.style={draw=cyan!90!teal!100!black},
	Bcol34/.style={draw=magenta!55!black}
}
\begin{subfigure}[t]{0.4\textwidth}
\centering
\begin{tikzpicture}[scale=1.2]
	\node[Avertex] (v1) at (0,1.4) {$v_1$};
	\node[Avertex] (v2) at (2,1.4) {$v_2$};
	\node[Avertex] (v3) at (0,0) {$v_3$};
	\node[Avertex] (v4) at (2,0) {$v_4$};
	\draw[Aedge,Acol12] (v1)--(v2); 
	\draw[Aedge,Acol12] (v1)--(v3);
	\draw[Aedge,Acol12] (v2)--(v4);
	\draw[Aedge,Acol12] (v3)--(v4);
	\draw[Bedge,Bcol12] (v1) to[bend left=12] (v4); 
	\draw[Bedge,Bcol12] (v2) to[bend right=12] (v3);
\end{tikzpicture}
\caption{$A_2$}
\end{subfigure}
\hfill
\begin{subfigure}[t]{0.58\textwidth}
\centering
\begin{tikzpicture}[scale=1.2]
	\foreach \k/\x in {1/0,2/1.9,3/3.8}{
		\node[Avertex] (t\k) at (\x,1.35) {$v_{1,\k}$};
		\node[Avertex] (b\k) at (\x,0) {$v_{2,\k}$};
	}
	\draw[Aedge,Acol12] (t1)--node[Alabel,above] {$a_{12}$} (t2);
	\draw[Aedge,Acol12] (b1)--(b2);
	\draw[Aedge,Acol13] (t1) to[bend left=22] node[Alabel,above] {$a_{13}$} (t3);
	\draw[Aedge,Acol13] (b1) to[bend right=22] (b3);
	\draw[Aedge,Acol23] (t2)--node[Alabel,above] {$a_{23}$} (t3);
	\draw[Aedge,Acol23] (b2)--(b3);
	\draw[Bedge,Bcol12] (t1) to[bend right=8] node[Alabel,sloped,above] {$b_{12}$} (b2);
	\draw[Bedge,Bcol12] (b1) to[bend left=8] (t2);
	\draw[Bedge,Bcol13] (t1) to[bend right=12] node[Alabel,sloped,above,pos=0.48] {$b_{13}$} (b3);
	\draw[Bedge,Bcol13] (b1) to[bend left=12] (t3);
	\draw[Bedge,Bcol23] (t2) to[bend right=8] node[Alabel,sloped,above] {$b_{23}$} (b3);
	\draw[Bedge,Bcol23] (b2) to[bend left=8] (t3);
\end{tikzpicture}
\caption{$A_3$}
\end{subfigure}
\par\medskip
\begin{subfigure}[t]{0.98\textwidth}
\centering
\begin{tikzpicture}[scale=1.15]
	\foreach \k/\x in {1/0,2/2.7,3/5.4,4/8.1}{
		\node[Avertex] (t\k) at (\x+1.5,1.85) {$v_{1,\k}$};
		\node[Avertex] (b\k) at (\x+1.5,0) {$v_{2,\k}$};
	}
	\draw[Aedge,Acol12] (t1)--node[Alabel,above] {$a_{12}$} (t2);
	\draw[Aedge,Acol12] (b1)--(b2);
	\draw[Aedge,Acol13] (t1) to[bend left=18] node[Alabel,above] {$a_{13}$} (t3);
	\draw[Aedge,Acol13] (b1) to[bend right=18] (b3);
	\draw[Aedge,Acol14] (t1) to[bend left=30] node[Alabel,above] {$a_{14}$} (t4);
	\draw[Aedge,Acol14] (b1) to[bend right=30] (b4);
	\draw[Aedge,Acol23] (t2)--node[Alabel,above] {$a_{23}$} (t3);
	\draw[Aedge,Acol23] (b2)--(b3);
	\draw[Aedge,Acol24] (t2) to[bend left=18] node[Alabel,above] {$a_{24}$} (t4);
	\draw[Aedge,Acol24] (b2) to[bend right=18] (b4);
	\draw[Aedge,Acol34] (t3)--node[Alabel,above] {$a_{34}$} (t4);
	\draw[Aedge,Acol34] (b3)--(b4);
	\draw[Bedge,Bcol12] (t1) to[bend right=7] node[Alabel,sloped,above] {$b_{12}$} (b2);
	\draw[Bedge,Bcol12] (b1) to[bend left=7] (t2);
	\draw[Bedge,Bcol13] (t1) to[bend right=12] node[Alabel,sloped,above,pos=0.48] {$b_{13}$} (b3);
	\draw[Bedge,Bcol13] (b1) to[bend left=12] (t3);
	\draw[Bedge,Bcol14] (t1) .. controls (-1.25,4.25) and (12.85,4.25) .. node[Alabel,above,pos=0.50] {$b_{14}$} (b4);
	\draw[Bedge,Bcol14] (b1) .. controls (-1.25,-2.40) and (12.85,-2.40) ..  (t4);
	\draw[Bedge,Bcol23] (t2) to[bend right=7] node[Alabel,sloped,above] {$b_{23}$} (b3);
	\draw[Bedge,Bcol23] (b2) to[bend left=7] (t3);
	\draw[Bedge,Bcol24] (t2) to[bend right=12] node[Alabel,sloped,above,pos=0.48] {$b_{24}$} (b4);
	\draw[Bedge,Bcol24] (b2) to[bend left=12] (t4);
	\draw[Bedge,Bcol34] (t3) to[bend right=7] node[Alabel,sloped,above] {$b_{34}$} (b4);
	\draw[Bedge,Bcol34] (b3) to[bend left=7] (t4);
\end{tikzpicture}
\caption{$A_4$}
\end{subfigure}
\caption{The gadgets $A_2$, $A_3$ and $A_4$. Edges in one color class are drawn with the same color and the same label; distinct color classes are drawn with distinct colors. }
\label{fig:A234}
\end{figure}

\begin{lemma}\label{L516}
	For every $m\geq 2$,  $A_m$ is $\mathcal{R}(K_m,2)$-saturated. For $m\geq 3$, every copy of $K_m$ in $A_m$   is rainbow.
\end{lemma}
\noindent {\bf Proof.}
The $m=2$ case is checked directly from the definition of $A_2$. Assume $m\geq 3$.
Let $J$ be a copy of $K_m$ in $A_m$. Assume $V(J)=\{v_{1,i_s}:1\le s\le l,i_s\in [m]\}\cup \{v_{2,j_s}:1\le s\le m-l,j_s\in [m]\}$, where $0\le l\le m$ ($l=0$ implies $V(J)=\{v_{2,j_s}:1\le s\le m,j_s\in [m]\}$). Then $\{i_1,\ldots,i_l\}\cap \{j_1,\ldots,j_{m-l}\}=\emptyset$. By the definition of $\mathcal{C}$, we easily have that $J$    is rainbow.
So we just need to show that $A_m$ is $\mathcal{R}(K_m,2)$-saturated when $m\ge 3$.

(a) Suppose there is a rainbow $K_m\cup K_2$ in $A_m$. If $V(K_2)\subseteq \{v_{j,k}: 1 \leq k \leq m\}$ for some $j \in \{1,2\}$, say $K_2=v_{1,1}v_{1,2}$. Then $v_{2,1},v_{2,2}\in V(K_m)$ by the construction. But $\mathcal{C}(v_{1,1}v_{1,2})=\mathcal{C}(v_{2,1}v_{2,2})$, a contradiction. Suppose $V(K_2)\not\subseteq \{v_{j,k}: 1 \leq k \leq m\}$ for any $j \in \{1,2\}$, say $K_2=v_{1,2}v_{2,1}$. Then $v_{1,1}v_{2,2}\in E(K_m)$. But $\mathcal{C}(v_{1,2}v_{2,1})=\mathcal{C}(v_{1,1}v_{2,2})$, a contradiction. Thus
 $A_m$ contains no rainbow copy of $K_m\cup K_2$.

(b) Let $v_{i,j}\in V(A_m)$, say $i=1$, and $c$ a color, where $1\le j\le m$.  If $c\not\in \{a_{l,k}:1 \leq l < k \leq m\}$, then $A_m[\{v_{2,k}:1 \leq  k \leq m\}]$ is a rainbow copy of $K_m$ avoiding both $v_{1,j}$ and  $c$. Assume $c=a_{l,k}$ for some $1 \leq l < k \leq m$. If $j=l$ (resp. $j=k$),
 then $A_m[\{v_{2,s}:1\le s\le m,s\not=k\}\cup\{v_{1,k}\}]$ (resp. $A_m[\{v_{2,s}:1\le s\le m,s\not=l\}\cup\{v_{1,l}\}]$) is a rainbow copy of $K_m$ avoiding both $v_{1,j}$ and $c$. If $j\notin \{l,k\}$, each of the two copies above is a required rainbow copy.


(c) Let $e\notin E(A_m)$, say $e=v_{1,r}v_{2,r}$, and $c$ a color. In the following, the second subscript is taken modulo $m$. Denote $A_m+_c e$ as $A'_m$.
If $c\notin X_{\mathcal{C}}$, then let $H=A'_m[\{v_{1,s}:1\le s\le m,s\not=r+1\}\cup\{v_{2,r}\}]\cup v_{2,r+1}v_{2,r+2}$.
Assume $c=a_{i,j}$ for some $1 \leq i < j \leq m$. If $r\in\{i,j\}$ (resp. $r\notin\{i,j\}$), let $h$ be the unique element of $\{i,j\}\setminus\{r\}$ (resp. $h=i$). Set $H=A'_m[\{v_{1,s}:1\le s\le m,s\not=h\}\cup\{v_{2,r}\}]\cup v_{2,h-1}v_{1,h}$ if $h-1\not=r$ or $H=A'_m[\{v_{1,s}:1\le s\le m,s\not=h\}\cup\{v_{2,r}\}]\cup v_{1,h}v_{2,h+1}$ if $h+1\not=r$. In each case, $H$ is a rainbow copy of $K_m\cup K_2$ in $A_m+_c e$.
Respectively assume $c=b_{i,j}$ for some $1 \leq i < j \leq m$. If $r\in\{i,j\}$ (resp. $r\notin\{i,j\}$), let $h$ be the unique element of $\{i,j\}\setminus\{r\}$ (resp. $h=i$). Set $H=A'_m[\{v_{1,s}:1\le s\le m,s\not=h\}\cup\{v_{2,r}\}]\cup v_{2,h-1}v_{2,h}$ if $h-1\not=r$ or $H=A'_m[\{v_{1,s}:1\le s\le m,s\not=h\}\cup\{v_{2,r}\}]\cup v_{2,h}v_{2,h+1}$ if $h+1\not=r$. In each case, $H$ is a rainbow copy of $K_m\cup K_2$ in $A_m+_c e$.

Thus $A_m$ is $\mathcal{R}(K_m,2)$-saturated. \qed


\vspace{0.3em}

\begin{prop}\label{PClique}
	For every $s\geq 2$ and $n\geq 2s+2$,
	\[
		rsat(n,K_s\cup K_2)\leq
		\begin{cases}
			6 & \text{if }s=2,\\
			\binom{2s}{2}-s & \text{if }s\geq 3.
		\end{cases}
	\]
\end{prop}

\noindent{\bf Proof.}
By Lemma~\ref{L516}, $A_s$ is $\mathcal{R}(K_s,2)$-saturated. Since $K_s$ is endpoint-less, Proposition~\ref{P513} implies that $A_s\cup (n-2s)K_1$ is $(K_s\cup K_2)$-rainbow saturated. For $s=2$ this gives the bound $rsat(n,2K_2) \le e(A_2)=6$. For $s\geq 3$,
\[
	rsat(n,K_s\cup K_2) \le e(A_s)=4\binom{s}{2}=\binom{2s}{2}-s.
\]
\qed
\vspace{0.5em}

\medskip\noindent {\bf Proof of Theorem~\ref{TIsolated}.}
The case $F=K_2$ is immediate since $rsat(n,K_2)=0$. Assume $F\not=K_2$. Then we have $F=H\cup K_2$ where $H$ has no isolated vertices. If $H$ is not complete, Proposition~\ref{PNonClique} supplies a constant upper bound; if $H=K_s$, Proposition~\ref{PClique} applies. In either case $rsat(n,F)=O(1)$ for all sufficiently large $n$. \qed
\vspace{0.5em}

\medskip\noindent {\bf Proof of Theorem~\ref{T11}.}
If $F$ contains an isolated edge, Theorem~\ref{TIsolated} gives $rsat(n,F)=O(1)$. If $F$ contains no isolated edge, Theorem~\ref{TLower} gives $rsat(n,F)=\Omega(n)$. Since $rsat(n,F)=O(n)$ for every graph $F$ by~\cite{BEH}, we conclude that $rsat(n,F)=\Theta(n)$ in the latter case. \qed

\section{Generalized friendship graphs}

The constructions in Section~3 use bounded colored gadgets to supply disjoint rainbow edges while avoiding a prescribed vertex or color. We now reuse this method to build the petal part of a generalized friendship graph, with a fixed clique serving as the common center. This section proves Theorem~\ref{TFriend}.

We first use the following theorem of Chakraborti, Hendrey, Lund and Tompkins to obtain the lower bound of Theorem~\ref{TFriend}.

\begin{theorem}\label{T41}{\bf \cite{CHA}}
	For every $s \geq 3$ and all sufficiently large $n$, we have
	$$
	rssat(n,K_s) = \left\{ \begin{aligned}
		& \; 2(n-2) \, , & & \text{if } s =3 \,, \\
		& \; (s-1)(n-s+1)+\binom{s-1}{2} \, , & \quad & \text{if } s \geq 4\,.
	\end{aligned}
	\right.
	$$
\end{theorem}
\vspace{0.1em}

Let $k=p+q$. Since every edge of $F_{t,p,q}=(tK_p) \vee K_q$ lies in a copy of $K_k$, every $F_{t,p,q}$-rainbow saturated graph is $K_k$-rainbow-semisaturated. Thus Theorem~\ref{T41}, applied with $s=k$, gives
\[
	rsat(n,F_{t,p,q})\ge rssat(n,K_k)=(k-1)n-O(1)=(p+q-1)n-O(1).
\]
This proves the lower bound of Theorem~\ref{TFriend}. It remains to exhibit an $F_{t,p,q}$-rainbow saturated graph with at most $(p+q-1)n+O(1)$ edges for each admissible $(t,p,q)$ where $t\ge 2$.

\FloatBarrier
\subsection{Core constructions}

We use the gadget graph $A_m$ from Definition~\ref{D516}. In each construction below the vertex set is partitioned into a center $C$, a bounded support $S$, an independent set $R$ with $N(r)=V(C)$ for every $r\in R$, and (when present) a copy of $B_t$. Unspecified edges are absent, and unspecified colors are fresh and pairwise distinct.

\medskip\noindent\textbf{The $p\geq3$ case.}

Let $k=p+q$. Take $C\cong K_{k-1}$ with vertex set $\{c_1,\ldots,c_{k-1}\}$, $S=(t-2)A_p \cup A_{p-1}$,
where $V(C)\cap V(S)=\emptyset$, and $R$  the remaining vertices which is an independent set. Let $G_1=C\vee (S\cup R)$ and each $A$-copy in $S$ uses a disjoint color set,
 the other edges of $G_1$ all receive pairwise distinct fresh colors (see Figure 2). Denote the $t-1$ components of $S$ by $A_p^1,\ldots,A_p^{t-2},A_{p-1}^{t-1}$. For each component $A_m^i$ in $S$ with $m\ge 3$, denote $V(A_m^i)=V_1^i\cup\cdots\cup V_m^i$ and the color set $\{a_{r,s}^i,b_{r,s}^i:1\le r<s\le m\}$, where $V_j^i=\{v_{1,j}^i,v_{2,j}^i\}$ for $1\le j\le m$ and $1\le i\le t-1$.

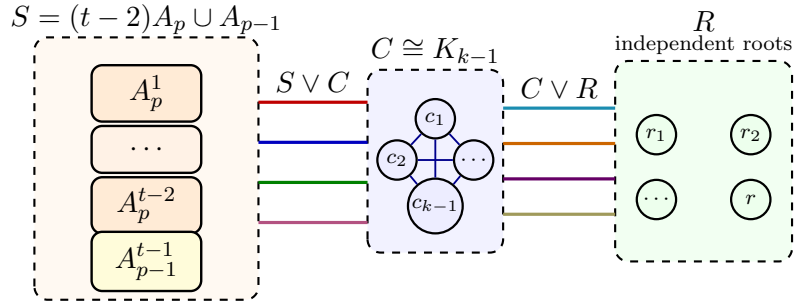
\begin{figure}[!ht]
\centering
\tikzset{
	casebox/.style={draw,dashed,rounded corners=5pt,line width=0.8pt},
	caseblock/.style={draw,rounded corners=4pt,minimum width=1.45cm,minimum height=0.62cm,font=\small,align=center,line width=0.9pt},
	casevertex/.style={circle,draw,fill=white,inner sep=0.8pt,minimum size=5.2mm,font=\scriptsize,line width=0.85pt},
	casejoin/.style={line width=1.2pt}
}
\begin{tikzpicture}[scale=0.78]
	\draw[casebox,fill=orange!5] (-5.35,-2.35) rectangle (-1.55,2.10);
	\node[font=\small] at (-3.45,2.41) {$S=(t-2)A_p\cup A_{p-1}$};
	\node[caseblock,fill=orange!16] (Apone) at (-3.45,1.15) {$A_p^1$};
	\node[caseblock,fill=orange!10] (Apdots) at (-3.45,0.20) {$\cdots$};
	\node[caseblock,fill=orange!16] (Aplast) at (-3.45,-0.75) {$A_p^{t-2}$};
	\node[caseblock,fill=yellow!18] (Apm) at (-3.45,-1.70) {$A_{p-1}^{t-1}$};

	\draw[casebox,fill=blue!5] (0.30,-1.55) rectangle (2.60,1.55);
	\node[font=\small] at (1.45,1.86) {$C\cong K_{k-1}$};
	\node[casevertex,fill=blue!7] (c1) at (1.45,0.72) {$c_1$};
	\node[casevertex,fill=blue!7] (c2) at (0.80,0.02) {$c_2$};
	\node[casevertex,fill=blue!7] (cd) at (2.10,0.02) {$\cdots$};
	\node[casevertex,fill=blue!7] (ck) at (1.45,-0.76) {$c_{k-1}$};
	\foreach \u/\v in {c1/c2,c1/cd,c1/ck,c2/cd,c2/ck,cd/ck}{
		\draw[line width=0.7pt,draw=blue!55!black] (\u)--(\v);
	}

	\draw[casebox,fill=green!5] (4.50,-1.70) rectangle (7.50,1.70);
	\node[font=\small] at (6.00,2.38) {$R$};
	\node[font=\scriptsize] at (6.00,1.98) {independent roots};
	\node[casevertex,fill=green!7] (r1) at (5.20,0.45) {$r_1$};
	\node[casevertex,fill=green!7] (r2) at (6.80,0.45) {$r_2$};
	\node[casevertex,fill=green!7] (rd) at (5.20,-0.65) {$\cdots$};
	\node[casevertex,fill=green!7] (rn) at (6.80,-0.65) {$r$};

	\draw[casejoin,draw=red!75!black] (-1.55,1.00)--(0.30,1.00);
	\draw[casejoin,draw=blue!75!black] (-1.55,0.32)--(0.30,0.32);
	\draw[casejoin,draw=green!50!black] (-1.55,-0.36)--(0.30,-0.36);
	\draw[casejoin,draw=magenta!70!black] (-1.55,-1.04)--(0.30,-1.04);
	\node[font=\small,fill=white,inner sep=1pt] at (-0.63,1.34) {$S\vee C$};

	\draw[casejoin,draw=cyan!65!black] (2.60,0.90)--(4.50,0.90);
	\draw[casejoin,draw=orange!80!black] (2.60,0.30)--(4.50,0.30);
	\draw[casejoin,draw=violet!80!black] (2.60,-0.30)--(4.50,-0.30);
	\draw[casejoin,draw=yellow!55!black] (2.60,-0.90)--(4.50,-0.90);
	\node[font=\small,fill=white,inner sep=1pt] at (3.55,1.24) {$C\vee R$};

\end{tikzpicture}
\caption{The construction of $G_1$, where $k=p+q$. The blocks $A_p$ and $A_{p-1}$ are shown as single components; the colored bundles indicate complete joins with fresh distinct colors. There are no edges between $S$ and $R$.}
\label{fig:case1}
\end{figure}

\begin{prop}\label{P519}
	 $G_1$  is $F_{t,p,q}$-rainbow saturated for $p \geq 3$ and has $(p+q-1)n+O(1)$ edges.
\end{prop}
\noindent {\bf Proof.} \textit{Rainbow-freeness.}
Suppose $J$ is a  rainbow copy of $F_{t,p,q}$ in $G_1$, with center $Q$ and petals $P_1,\ldots,P_t$. Then $Q\cong K_q$ and $P_i\cong K_p$ for $1\le i\le t$. Note that $d_J(v)=tp+q-1$ for any $v\in V(Q)$.
We first have the following claim.

\vskip.2cm

\noindent{\bf Claim 1} For any $P_i,P_j\in \{P_1,\ldots,P_t\}$ with $i\not=j$ and any $1\le l\le t-1$, we have
\begin{equation}\label{localbound}
	\bigl|(V(P_i)\cup V(P_j))\cap V(A_m^l)\bigr|\le m+1,
\end{equation} where $m=p$ for $1\le l\le t-2$ and $m=p-1$ for $l=t-1$.


\noindent{\bf Proof of Claim 1~} When $m=p-1$ and $p=3$, we easily have $|(V(P_i)\cup V(P_j))\cap V(A_2^l)|\le 3$. Suppose there are $P_i,P_j\in \{P_1,\ldots,P_t\}$, say $i=1,j=2$, such that $|(V(P_1)\cup V(P_2))\cap V(A_m^l)|\ge m+2$, where $1\le l\le t-1$.
Then  there are $V_r^l,V_s^l\subseteq V(A_m^l)$ such that $V_r^l\cup V_s^l\subseteq V(P_1)\cup V(P_2)$, where $1\le r,s\le m$. Assume $r<s$. Then $J[V(P_1)\cup V(P_2)]$ has two edges which are colored by  $a^l_{r,s}$ or $b^l_{r,s}$, a contradiction.\q
\vskip.2cm


Since $d(r)=k-1=p+q-1<tp+q-1$ for every $r\in R$, we have $V(Q)\cap R=\emptyset$. We claim $V(Q)\subseteq V(C)$. Suppose $Q$ meets a block $A_m^i$ ($m\in\{p-1,p\}$, $1\le i\le t-1$). Then $V(J)\subseteq V(A_m^i)\cup V(C)$. So
\[
	tp+q\le |V(A_m^i)\cup V(C)|=2m+p+q-1\le 3p+q-1
\]
which implies $t=2$. Then we have $i=1$, $m=p-1$ and every petal meets $A_{p-1}^1$.
When $p\ge 4$, assume, without loss of generality, that $V(Q)\cap V(A_{p-1}^1)=\{v_{1,1}^1,\ldots,v_{1,s}^1,v_{2,s+1}^1,\ldots,v_{2,s+l}^1\}$, where $s,l\ge 0$ and $1\le s+l\le p-1$. If $ s+l= p-1$, then $V(P_1)\cup V(P_2)\subseteq V(C)\setminus V(Q)$ which implies $2p\le q+p-1-(q-p+1)$, a contradiction. If $s+l=p-2$, then $|(V(P_1)\cup V(P_2))\cap V(A_{p-1}^1)|\le 2$ which implies $|(V(P_1)\cup V(P_2))\cap (V(C)\setminus V(Q))|\ge 2p-2$. But $|V(C)\setminus V(Q)|=2p-3$, a contradiction. If $s+l=p-3$, then $|(V(P_1)\cup V(P_2))\cap V(A_{p-1}^1)|\le 3$ by the construction of $A_{p-1}^1$ which implies $|(V(P_1)\cup V(P_2))\cap (V(C)\setminus V(Q))|\ge 2p-3$. But $|V(C)\setminus V(Q)|=2p-4$, a contradiction. So we have $1\le s+l\le p-4$.

Denote $G_1[\cup_{r=s+l+1}^{p-1} V_r^1]=A'$. Then $A'$ is a subgraph of $A_{p-1-(s+l)}$. We have $V(P_1)\cup V(P_2)\subseteq V(A')\cup (V(C)\setminus V(Q))$. By Claim 1, $|(V(P_1)\cup V(P_2))\cap V(A')|\le p-(s+l)$. This implies $|(V(P_1)\cup V(P_2))\cap (V(C)\setminus V(Q))|\ge p+(s+l)$. But $|V(C)\setminus V(Q)|=p+q-1-(q-(l+s))=p+(s+l)-1$, a contradiction.

When $p=3$, by $J[V(J)\cap V(A_2^1)]$ being rainbow, $J[V(J)\cap V(A_2^1)]$ contains no $K_3$ or $2K_2$. However, each induced subgraph with four vertices in $J$ contains a copy of $K_3$ or $2K_2$, so we have $|V(J)\cap V(A_2^1)|\leq 3$. Then $|V(J)|\leq 3+|V(C)|= q+5< q+2p$, a contradiction. Hence $V(Q)\subseteq V(C)$.

Set $D=V(C)\setminus V(Q)$. Then $|D|=p-1$. Since $V(Q)\subseteq V(C)$, for each $1\le i\le t$, we have $V(P_i)\setminus V(C)\subseteq V(A_m^j)$ for some component $A_m^j$ in $S$, or $V(P_i)\setminus V(C)\subseteq R$ and $|V(P_i)\setminus V(C)|=1$. Suppose there are two petals, say $P_1$ and $P_2$, such that $V(P_i)\cap V(A^j_m)\not=\emptyset$, where $i=1,2$, $1\le j\le t-1$ and $m\in\{p-1,p\}$. By (\ref{localbound}), we have$$p-1\ge |(V(P_1)\cup V(P_2))\cap D|\ge 2p-(m+1).$$ Then we have $m=p$ and $|(V(P_1)\cup V(P_2))\cap D|=p-1$ which implies $D\subseteq V(P_1)\cup V(P_2)$ and $t\ge 3$. Assume without loss of generality that $j=1$. Then  $(V(P_i)\setminus V(C))\cap V(A_{p-1}^{t-1})=\emptyset$ by $A_{p-1}^{t-1}$ being rainbow-$K_p$-free and $(V(P_i)\setminus V(C))\cap R=\emptyset$ by $p\ge 3$ for $3\le i\le t$.  So the remaining $t-3$ copies of $A_p$ contribute at most one petal each, giving at most $t-1$ petals, a contradiction.

Note that if there is
 $1\le i\le t$ such that $(V(P_i)\setminus V(C))\cap V(A_{p-1}^{t-1})\not=\emptyset$, then $(V(P_j)\setminus V(C))\cap R=\emptyset$ for any $1\le j\le t$ and vice versa.  Since every component $A_m^i$ in $S$ meets at most one petal, $G_1$ is $F_{t,p,q}$-rainbow free.

\textit{Saturation.}
Put $H_0=(t-2)K_p\cup K_{p-1}$. By Lemmas~\ref{L515} and~\ref{L516}, $S=(t-2)A_p \cup A_{p-1}$ is $\mathcal{R}(H_0,2)$-semisaturated. Fix a nonedge $e$ of $G_1[V(S)\cup R]$ and a color $\alpha$.
Since $S=(t-2)A_p \cup A_{p-1}$ is $\mathcal{R}(H_0,2)$-semisaturated, it is not difficult to check that  there is a rainbow copy $K$ of $H_0\cup K_2$ in $G_1[V(S)\cup R]+_\alpha e$ containing $e$. Now we find a rainbow copy of $F_{t,p,q}$ in $G_1+_\alpha e$.  Denote $K=P_1\cup\ldots\cup P_{t-2}\cup L\cup M$, where $P_i\cong K_p$ for $1\le i\le t-2$, $L\cong K_{p-1}$ and $M\cong K_2$. We partition
$V(C)=D_1\cup D_2\cup Q$ such that $|D_1|=1$, $ |D_2|=p-2$ and $ |Q|=q$ as follows:
\begin{itemize}

\item if all edges in $E(V(S)\cup R,V(C))\cup E(G_1[C])$ are not colored by $\alpha$, we arbitrarily partition $V(C)$;

\item if there is $cz\in E(G_1)$ with $c\in V(C)$ and $z\notin V(C)$ which is colored by $\alpha$, then let $c\in D_1$ (resp. $c\in D_2$) when $z\notin V(L)$ (resp. $z\notin V(M)$);

\item if there is $c_1c_2\in E(G_1)$ with $c_1,c_2\in V(C)$ which is colored by $\alpha$, then  put $c_1\in D_1$ and $c_2\in D_2$.
\end{itemize}
In each case, $(P_1\cup\ldots\cup P_{t-2}\cup (L\vee D_1)\cup (M\vee C[D_2]))\vee C[Q]$ is a rainbow copy of $F_{t,p,q}$ in $G_1+_\alpha e$. One can check that $e(G_1)=(p+q-1)n+O(1)$.
\qed
\vspace{0.3em}

\FloatBarrier
\medskip\noindent\textbf{The $p=2$ and $q\geq3$ case.}

Let $C,S\cong K_{q+1}$ with vertex set $\{c_1,\ldots,c_{q+1}\}$ and $\{s_1,\ldots,s_{q+1}\}$, respectively and  $R$  the remaining vertices which is an independent set.
Let $G_2=C\vee (S\cup R)$ and
color the edges so that $\gamma_{ij}:=\mathcal{C}(s_is_j)=\mathcal{C}(c_ic_j)$ for each $1\leq i<j\leq q+1$, and give every other edge a fresh and distinct color (see Figure 3).

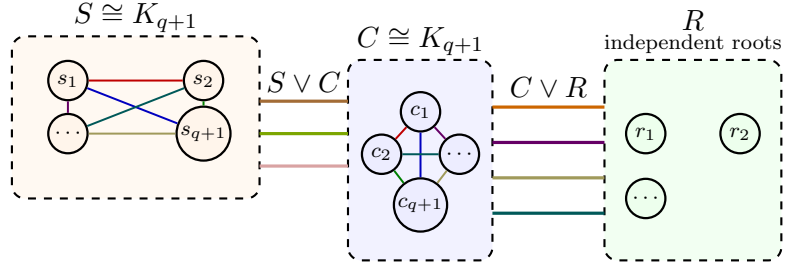
\begin{figure}[!ht]
	\centering
	\tikzset{
		casebox/.style={draw,dashed,rounded corners=5pt,line width=0.8pt},
		caseblock/.style={draw,rounded corners=4pt,minimum width=1.45cm,minimum height=0.62cm,font=\small,align=center,line width=0.9pt},
		casevertex/.style={circle,draw,fill=white,inner sep=0.8pt,minimum size=5.2mm,font=\scriptsize,line width=0.85pt},
		casejoin/.style={line width=1.2pt}
	}
	\begin{tikzpicture}[scale=0.78]
		\draw[casebox,fill=orange!5] (-5.55,-0.65) rectangle (-1.35,2.10);
		\node[font=\small] at (-3.45,2.42) {$S\cong K_{q+1}$};
		\node[casevertex,fill=orange!8] (s1) at (-4.60,1.35) {$s_1$};
		\node[casevertex,fill=orange!8] (s2) at (-2.30,1.35) {$s_2$};
		\node[casevertex,fill=orange!8] (sd) at (-4.60,0.45) {$\cdots$};
		\node[casevertex,fill=orange!8] (sq) at (-2.30,0.45) {$s_{q+1}$};
		\draw[line width=0.75pt,draw=red!75!black] (s1)--(s2);
		\draw[line width=0.75pt,draw=violet!80!black] (s1)--(sd);
		\draw[line width=0.75pt,draw=blue!75!black] (s1)--(sq);
		\draw[line width=0.75pt,draw=teal!70!black] (s2)--(sd);
		\draw[line width=0.75pt,draw=green!50!black] (s2)--(sq);
		\draw[line width=0.75pt,draw=yellow!55!black] (sd)--(sq);

		\draw[casebox,fill=blue!5] (0.15,-1.70) rectangle (2.60,1.70);
		\node[font=\small] at (1.38,2.02) {$C\cong K_{q+1}$};
		\node[casevertex,fill=blue!7] (c1) at (1.38,0.82) {$c_1$};
		\node[casevertex,fill=blue!7] (c2) at (0.72,0.10) {$c_2$};
		\node[casevertex,fill=blue!7] (cd) at (2.04,0.10) {$\cdots$};
		\node[casevertex,fill=blue!7] (cq) at (1.38,-0.76) {$c_{q+1}$};
		\draw[line width=0.75pt,draw=red!75!black] (c1)--(c2);
		\draw[line width=0.75pt,draw=violet!80!black] (c1)--(cd);
		\draw[line width=0.75pt,draw=blue!75!black] (c1)--(cq);
		\draw[line width=0.75pt,draw=teal!70!black] (c2)--(cd);
		\draw[line width=0.75pt, draw=green!50!black] (c2)--(cq);
		\draw[line width=0.75pt,draw=yellow!55!black] (cd)--(cq);
	
		\draw[casebox,fill=green!5] (4.50,-1.70) rectangle (7.50,1.70);
		\node[font=\small] at (6.00,2.38) {$R$};
		\node[font=\scriptsize] at (6.00,1.98) {independent roots};
		\node[casevertex,fill=green!7] (r1) at (5.20,0.45) {$r_1$};
		\node[casevertex,fill=green!7] (r2) at (6.80,0.45) {$r_2$};
		\node[casevertex,fill=green!7] (rd) at (5.20,-0.65) {$\cdots$};
	
		\draw[casejoin,draw=brown!85!black] (-1.35,1.00)--(0.15,1.00);
		\draw[casejoin,draw=lime!65!black] (-1.35,0.45)--(0.15,0.45);
		\draw[casejoin,draw=pink!85!black] (-1.35,-0.10)--(0.15,-0.10);
		\node[font=\small,fill=white,inner sep=1pt] at (-0.60,1.32) {$S\vee C$};
	
		\draw[casejoin,draw=orange!80!black] (2.60,0.90)--(4.50,0.90);
		\draw[casejoin,draw=violet!80!black] (2.60,0.30)--(4.50,0.30);
		\draw[casejoin,draw=yellow!55!black] (2.60,-0.30)--(4.50,-0.30);
		\draw[casejoin,draw=teal!70!black] (2.60,-0.90)--(4.50,-0.90);
		\node[font=\small,fill=white,inner sep=1pt] at (3.55,1.24) {$C\vee R$};
	
	\end{tikzpicture}
	\caption{The construction of $G_2$. Equal colors on $S$ and $C$ indicate $\gamma_{ij}$. There are no edges between $S$ and $R$.}
	\label{fig:case2}
\end{figure}

\begin{prop}\label{P520}
	 $G_2$  is $F_{2,2,q}$-rainbow saturated for $q \geq 3$.
\end{prop}
\noindent {\bf Proof.} Put $I=\{1,\ldots,q+1\}$, $C_X=\{c_i:i\in X\}$ and $S_X=\{s_i:i\in X\}$, where $X\subseteq I$.

\textit{Rainbow-freeness.}
Suppose $J$ is a rainbow copy of $F_{2,2,q}=2K_2\vee K_q$ in $G_2$ with center $Q$ and petals $P_1,P_2$, where $Q\cong K_q$ and $P_i\cong K_2$ for $i=1,2$. Then $V(Q)\cap R=\emptyset$; otherwise  the remaining $q + 3$ vertices of $J$ would lie in $V(C)$, a contradiction. Write $V(Q)=S_A\cup C_B$ with $A,B\subseteq I$.

Suppose $A=\emptyset$. Assume $V(Q)=C_{I\setminus\{q+1\}}$. Since $J$ is rainbow and $|(V(P_1)\cup V(P_2))\cap V(S)|\ge 2$, by the colors $\gamma_{ij}$, we have $|(V(P_1)\cup V(P_2))\cap V(S)|= 2$ and $s_{q+1}\in V(P_1)\cup V(P_2)$, say $s_{q+1}\in V(P_1)$. Let $P_1=s_{q+1}s_i$, where $i\in I\setminus\{q+1\}$. Then $c_{q+1}\in V(P_2)$. But $s_{q+1}s_i,c_{q+1}c_i\in E(J)$ with the same color, a contradiction.
Thus $A\neq\emptyset$. By the same argument, $B\neq\emptyset$. Also $|A\cap B|\le 1$; otherwise two repeated indices would give two center edges with the same color.
If there exists $i\in A\cap B$, then $|A\cup B|=q-1$ and $|I\setminus(A\cup B)|=2$. For each $a\in(A\cup B)\setminus\{i\}$, the color $\gamma_{ai}$ already appears on a center edge, so neither $s_a$ (when $a\in B$) nor $c_a$ (when $a\in A$) can lie in $V(J)\setminus V(Q)$.
Each $j\in I\setminus(A\cup B)$ contributes at most one of $s_j,c_j$; otherwise  $s_js_i,c_jc_i\in E(J)$ with the color $\gamma_{ji}$.
Hence at most two usable petal endpoints remain in $V(S)\cup V(C)$, a contradiction. Therefore $A\cap B=\emptyset$, so $|A|+|B|=q$ and $I\setminus(A\cup B)=\{\ell\}$ for a unique $\ell$.

The endpoints of every petal of $J$ therefore are in $S_B\cup C_A\cup\{s_\ell,c_\ell\}$.
If both vertices of a petal edge lies in $S_B$ or $C_A$, its color would coincident with the corresponding edge used in $C_B$ or $S_A$, which is impossible. Cross-edges $s_xc_y$ ($x\in B$, $y\in A$) are unusable, since the joins $s_xs_y$ and $c_yc_x$ would both receive color $\gamma_{xy}$.
Thus both petals meet $\{s_\ell,c_\ell\}$, and there are only two types of disjoint pairs. If the disjoint pairs are $\{s_\ell s_x,\ c_\ell c_y\} \text{ and } \{s_\ell c_y,\ c_\ell s_x\}$
($x\in B$, $y\in A$), they force the repeated colors $\gamma_{\ell x}$ and $\gamma_{xy}$ respectively, a contradiction.
If the disjoint pairs are $\{s_\ell s_x,\ c_\ell s_y\} \text{ and } \{s_\ell c_y,\ c_\ell c_x\}$ ($x,y\in B$ or $x,y\in A$ respectively), they force repeated colors between a connection edge ($c_\ell c_x$ or $s_ell s_x$) and a petal edge ($s_ell s_x$ or $c_\ell c_x$), a contradiction.

\textit{Saturation.}
Fix a nonedge $e$ of $G_2$ and a color $\alpha$. Assume  $e\in\{r_1r_2,r_1s_1\}$, where $r_1,r_2\in R$. For  $1\le i\neq j\le q+1$, set
\begin{align*}
	\Phi_{ij}&=(r_1r_2\cup s_is_j)\vee(C[V(C)\setminus\{c_j\}]),\\
	\Omega_{ij}&=(r_1s_1\cup s_ic_j)\vee(C[V(C)\setminus\{c_j\}])~~\mbox{where}~~i\not=1,\\
	\Psi_{ij}&=(r_1s_1\cup s_is_j)\vee(C[V(C)\setminus\{c_j\}])~~\mbox{where}~~1\notin\{i,j\}.
\end{align*}
Each $\Phi_{ij}$ (respectively $\Omega_{ij}$ and $\Psi_{ij}$) is a copy of $2K_2\vee K_q$ in $G_2+ e$ containing $e$ when $e=r_1r_2$ (respectively when $e=r_1s_1$). If $\alpha$ does not appear in $G_2$, every such copy is a rainbow copy of $2K_2\vee K_q$ in $G_2+_\alpha e$. Thus it remains to treat $\alpha$ already used on $G_2$. Note that $|I|\ge4$. If $e=r_1r_2$, choose $\Phi_{ij}$ as follows: $j=a$, $i\notin\{a,b\}$ when $\alpha=\gamma_{ab}$; $i,j\in I\setminus\{a\}$ when $\alpha=\mathcal{C}(s_ac_b)$; $j=b$ when $\alpha=\mathcal{C}(r_hc_b)$ ($r_h\in R$). If $e=r_1s_1$, choose $\Psi_{ij}$ by $j\in\{a,b\}\setminus\{1\}$, $i\notin\{1,a,b\}$ when $\alpha=\gamma_{ab}$; $i,j\in I\setminus\{1,a\}$ when $\alpha=\mathcal{C}(s_ac_b)$ with $a\neq1$; $j=b$, $i\notin\{1,b\}$ when $a=1$ or $\alpha=\mathcal{C}(r_1c_b)$ with $b\ne 1$, and choose arbitrarily when $\alpha=\mathcal{C}(r_hc_b)$ with $h\ne 1$. Finally, if $\alpha\in\{\mathcal{C}(s_1c_1),\mathcal{C}(r_1c_1)\}$, use $\Omega_{21}$. In each case, there is a rainbow copy of $2K_2\vee K_q$ in $G_2+_\alpha e$.
\qed
\vspace{0.3em}

\newcommand{\casefigqgeqthree}{%
\begin{figure}[!ht]
\centering
\tikzset{
	casebox/.style={draw,dashed,rounded corners=5pt,line width=0.8pt},
	caseblock/.style={draw,rounded corners=4pt,minimum width=1.45cm,minimum height=0.62cm,font=\small,align=center,line width=0.9pt},
	casevertex/.style={circle,draw,fill=white,inner sep=0.8pt,minimum size=5.2mm,font=\scriptsize,line width=0.85pt},
	casejoin/.style={line width=1.2pt}
}
\begin{tikzpicture}[scale=0.78]
	\draw[casebox,fill=orange!5] (-5.55,-0.65) rectangle (-1.35,2.10);
		\node[font=\small] at (-3.45,2.42) {$S\cong K_{q+1}$};
		\node[casevertex,fill=orange!8] (s1) at (-4.60,1.35) {$s_1$};
	\node[casevertex,fill=orange!8] (s2) at (-2.30,1.35) {$s_2$};
	\node[casevertex,fill=orange!8] (sd) at (-4.60,0.45) {$\cdots$};
	\node[casevertex,fill=orange!8] (sq) at (-2.30,0.45) {$s_{q+1}$};
	\draw[line width=0.75pt,draw=red!75!black] (s1)--(s2);
	\draw[line width=0.75pt,draw=violet!80!black] (s1)--(sd);
	\draw[line width=0.75pt,draw=blue!75!black] (s1)--(sq);
	\draw[line width=0.75pt,draw=teal!70!black] (s2)--(sd);
	\draw[line width=0.75pt,draw=green!50!black] (s2)--(sq);
	\draw[line width=0.75pt,draw=yellow!55!black] (sd)--(sq);
	\node[caseblock,fill=yellow!16,minimum width=2.25cm] (Bt) at (-3.45,-1.33) {$B_t$};
	\node[font=\scriptsize] at (-3.45,-2.1) {added for $t\geq3$};

	\draw[casebox,fill=blue!5] (0.15,-1.70) rectangle (2.60,1.70);
	\node[font=\small] at (1.38,2.02) {$C\cong K_{q+1}$};
	\node[casevertex,fill=blue!7] (c1) at (1.38,0.82) {$c_1$};
	\node[casevertex,fill=blue!7] (c2) at (0.72,0.10) {$c_2$};
	\node[casevertex,fill=blue!7] (cd) at (2.04,0.10) {$\cdots$};
	\node[casevertex,fill=blue!7] (cq) at (1.38,-0.76) {$c_{q+1}$};
	\draw[line width=0.75pt,draw=red!75!black] (c1)--(c2);
	\draw[line width=0.75pt,draw=violet!80!black] (c1)--(cd);
	\draw[line width=0.75pt,draw=blue!75!black] (c1)--(cq);
	\draw[line width=0.75pt,draw=teal!70!black] (c2)--(cd);
	\draw[line width=0.75pt, draw=green!50!black] (c2)--(cq);
	\draw[line width=0.75pt,draw=yellow!55!black] (cd)--(cq);

	\draw[casebox,fill=green!5] (4.50,-1.70) rectangle (7.50,1.70);
	\node[font=\small] at (6.00,2.38) {$R$};
	\node[font=\scriptsize] at (6.00,1.98) {independent roots};
	\node[casevertex,fill=green!7] (r1) at (5.20,0.45) {$r_1$};
	\node[casevertex,fill=green!7] (r2) at (6.80,0.45) {$r_2$};
	\node[casevertex,fill=green!7] (rd) at (5.20,-0.65) {$\cdots$};
	\node[casevertex,fill=green!7] (rn) at (6.80,-0.65) {$r$};

	\draw[casejoin,draw=brown!85!black] (-1.35,1.00)--(0.15,1.00);
	\draw[casejoin,draw=lime!65!black] (-1.35,0.45)--(0.15,0.45);
	\draw[casejoin,draw=pink!85!black] (-1.35,-0.10)--(0.15,-0.10);
	\node[font=\small,fill=white,inner sep=1pt] at (-0.60,1.32) {$S\vee C$};
	\draw[casejoin,draw=magenta!70!black] ([yshift=3pt]Bt.east)--(0.15,-1.20);
	\draw[casejoin,draw=cyan!65!black] ([yshift=-5pt]Bt.east)--(0.15,-1.50);
	\node[font=\scriptsize,fill=white,inner sep=1pt] at (-0.85,-0.92) {$B_t\vee C$};

	\draw[casejoin,draw=orange!80!black] (2.60,0.90)--(4.50,0.90);
	\draw[casejoin,draw=violet!80!black] (2.60,0.30)--(4.50,0.30);
	\draw[casejoin,draw=yellow!55!black] (2.60,-0.30)--(4.50,-0.30);
	\draw[casejoin,draw=teal!70!black] (2.60,-0.90)--(4.50,-0.90);
	\node[font=\small,fill=white,inner sep=1pt] at (3.55,1.24) {$C\vee R$};

\end{tikzpicture}
\caption{The graph $H_t$ obtained from $G_2$ by adjoining $B_t$ and the complete join $B_t\vee C$. There are no edges between any two of the sets $S$, $V(B_t)$, and $R$.}
\label{fig:Gt2}
\end{figure}
}

\medskip\noindent\textbf{The $p=2$ and $q=2$ case.}

Let $C\cong K_3$ with vertex set $\{c_1,c_2,c_3\}$, $S\cong K_6$ with vertex $\{s_1,\ldots,s_6\}$, and let $R$ be the remaining vertices which is an independent set (assume $R\neq\emptyset$).
Let  $G_3=C\vee (S\cup R)$ and
color the edges of $G_3$ so that
\[
\begin{aligned}
	\gamma_{ij}&:=\mathcal{C}(c_ic_j)=\mathcal{C}(s_is_j)=\mathcal{C}(s_{i+3}s_{j+3})=\mathcal{C}(s_ks_{k+3}),
	&& \quad \text{for }\{i,j,k\}=\{1,2,3\},\\
	\delta_{ij}&:=\mathcal{C}(s_ic_j)=\mathcal{C}(c_is_{j+3})=\mathcal{C}(s_is_{j+3}),
	&&\quad\text{for }1\le i\neq j\le 3.
\end{aligned}
\]
(where $\gamma_{ij}=\gamma_{ji}$) and give every other edge a fresh and distinct color  (see Figure 4).

\begin{figure}[!ht]
	\centering
	\tikzset{
		casebox/.style={draw,dashed,rounded corners=5pt,line width=0.8pt},
		caseblock/.style={draw,rounded corners=4pt,minimum width=1.45cm,minimum height=0.62cm,font=\small,align=center,line width=0.9pt},
		casevertex/.style={circle,draw,fill=white,inner sep=0.8pt,minimum size=5.2mm,font=\scriptsize,line width=0.85pt},
		casejoin/.style={line width=1.2pt}
	}
	\definecolor{gammathirteencolor}{RGB}{90,110,145}
	\begin{tikzpicture}[scale=0.76]
		\draw[casebox,fill=orange!5] (-5.75,-1.15) rectangle (-1.15,2.30);
		\node[font=\small] at (-3.45,2.62) {$S\cong K_6$};
		\coordinate (sone) at (-4.85,1.20);
		\coordinate (stwo) at (-3.45,1.60);
		\coordinate (sthree) at (-2.05,1.20);
		\coordinate (sfour) at (-4.85,-0.10);
		\coordinate (sfive) at (-3.45,-0.50);
		\coordinate (ssix) at (-2.05,-0.10);
		\draw[line width=0.55pt,draw=red!75!black] (sone)--(stwo);
		\draw[line width=0.55pt,draw=red!75!black] (sfour)--(sfive);
		\draw[line width=0.55pt,draw=red!75!black] (sthree)--(ssix);
		\draw[line width=0.55pt,draw=blue!75!black] (stwo)--(sthree);
		\draw[line width=0.55pt,draw=blue!75!black] (sfive)--(ssix);
		\draw[line width=0.55pt,draw=blue!75!black] (sone)--(sfour);
		\draw[line width=0.55pt,draw=gammathirteencolor] (sone)--(sthree);
		\draw[line width=0.55pt,draw=gammathirteencolor] (sfour)--(ssix);
		\draw[line width=0.55pt,draw=gammathirteencolor] (stwo)--(sfive);
		\draw[line width=0.55pt,draw=brown!85!black] (sone)--(sfive);
		\draw[line width=0.55pt,draw=teal!70!black] (sone)--(ssix);
		\draw[line width=0.55pt,draw=black!85] (stwo)--(sfour);
		\draw[line width=0.55pt,draw=lime!65!black] (stwo)--(ssix);
		\draw[line width=0.55pt,draw=pink!85!black] (sthree)--(sfour);
		\draw[line width=0.55pt,draw=brown!45!yellow!100] (sthree)--(sfive);
		\node[casevertex,fill=orange!8] at (sone) {$s_1$};
		\node[casevertex,fill=orange!8] at (stwo) {$s_2$};
		\node[casevertex,fill=orange!8] at (sthree) {$s_3$};
		\node[casevertex,fill=orange!8] at (sfour) {$s_4$};
		\node[casevertex,fill=orange!8] at (sfive) {$s_5$};
		\node[casevertex,fill=orange!8] at (ssix) {$s_6$};
	
		\draw[casebox,fill=blue!5] (0.20,-1.35) rectangle (2.55,1.55);
		\node[font=\small] at (1.38,1.86) {$C\cong K_3$};
		\node[casevertex,fill=blue!7] (c1) at (1.38,0.82) {$c_1$};
		\node[casevertex,fill=blue!7] (c2) at (0.76,-0.28) {$c_2$};
		\node[casevertex,fill=blue!7] (c3) at (2.00,-0.28) {$c_3$};
		\draw[line width=0.75pt,draw=red!75!black] (c1)--(c2);
		\draw[line width=0.75pt,draw=gammathirteencolor] (c1)--(c3);
		\draw[line width=0.75pt,draw=blue!75!black] (c2)--(c3);
	
		\draw[casebox,fill=green!5] (4.50,-1.70) rectangle (7.50,1.70);
		\node[font=\small] at (6.00,2.38) {$R$};
		\node[font=\scriptsize] at (6.00,1.98) {independent roots};
		\node[casevertex,fill=green!7] (r1) at (5.20,0.45) {$r_1$};
		\node[casevertex,fill=green!7] (r2) at (6.80,0.45) {$r_2$};
		\node[casevertex,fill=green!7] (rd) at (5.20,-0.65) {$\cdots$};
	
		\draw[casejoin,draw=brown!85!black] (-1.15,0.80)--(0.20,0.80);
		\draw[casejoin,draw=lime!65!black] (-1.15,0.25)--(0.20,0.25);
		\draw[casejoin,draw=pink!85!black] (-1.15,-0.30)--(0.20,-0.30);
		\node[font=\small,fill=white,inner sep=1pt] at (-0.48,1.12) {$S\vee C$};

		\draw[casejoin,draw=orange!80!black] (2.55,0.80)--(4.50,0.80);
		\draw[casejoin,draw=violet!80!black] (2.55,0.20)--(4.50,0.20);
		\draw[casejoin,draw=yellow!55!black] (2.55,-0.40)--(4.50,-0.40);
		\node[font=\small,fill=white,inner sep=1pt] at (3.53,1.12) {$C\vee R$};
	
	\end{tikzpicture}
	\caption{The  construction of $G_3$. Equal colors in $S$ and $C$ represent the classes $\gamma_{12}$, $\gamma_{13}$, and $\gamma_{23}$. The three displayed $S$--$C$ lines have colors $\delta_{12}$, $\delta_{23}$, and $\delta_{31}$. There are no edges between $S$ and $R$.}
	\label{fig:case3}
\end{figure}

\begin{prop}\label{P521}
	 $G_3$  is $F_{2,2,2}$-rainbow saturated.
\end{prop}
\noindent {\bf Proof.} Write
\[
	M_{ij}=\{s_is_j,\ s_{i+3}s_{j+3},\ s_ks_{k+3}\}\qquad(\{i,j,k\}=\{1,2,3\}).
\]
Then $M_{ij}=M_{ji}$ and each $M_{ij}$ is the monochromatic $\gamma_{ij}$-matching.

\textit{Rainbow-freeness.}
Suppose $J$ is a rainbow copy of $F_{2,2,2}=2K_2\vee K_2$  with center $Q$ and  petals $P_1,P_2$. Then $V(Q)\cap R=\emptyset$ which implies $(V(P_1)\cup V(P_2))\cap V(S)\not=\emptyset$. Assume $V(P_1)\cap V(S)\not=\emptyset$.

Suppose first that $V(Q)\subseteq V(C)$, say $Q=c_1c_2$. We can assume $P_1\in E(S)$. Since $J$ is rainbow, $E(J)\cap M_{12}=\emptyset$. Denote
$$E'=\{s_1s_5,s_1s_6,s_2s_4,s_2s_6,s_3s_4,s_3s_5\}.$$
By the coloring of $\delta_{ij}$ with $1\le i\not=j\le 3$, we have $P_1,P_2\notin E'$.
Consider first $c_3\in V(J)$. Then  $\gamma_{12}$, $\gamma_{13}$ and $\gamma_{23}$ all appear in $J$ which implies $E(J)\cap (M_{13}\cup M_{23})=\emptyset$. So we have $P_1\in E'$, a contradiction. Assume $c_3\notin V(J)$. Then $V(P_1)\cup V(P_2)\subseteq V(S)$. If $s_1\in V(J)$ (resp. $s_2\in V(J)$), by the $\delta_{12}$ (resp. $\delta_{21}$), we have $s_5\notin V(J)$ (resp. $s_4\not\in V(J)$). Since $P_1,P_2\notin E'$ and $E(J)\cap M_{12}=\emptyset$, we have $|V(J)\cap \{s_1,s_2\}|\ge 1$. If $|V(J)\cap \{s_1,s_2\}|= 2$, then $V(P_1)\cup V(P_2)=\{s_1,s_2,s_3,s_6\}$. But $s_1s_2,s_1s_6,s_2s_6\notin E(J)$, a contradiction. If $s_1\in V(J)$ (resp. $s_2\in V(J)$), then $V(P_1)\cup V(P_2)=\{s_1,s_3,s_4,s_6\}$ (resp. $V(P_1)\cup V(P_2)=\{s_2,s_3,s_5,s_6\}$). Since $P_1,P_2\notin E'$ and $E(J)\cap M_{12}=\emptyset$, we have $s_1s_3,s_4s_6\in E(J)$ (if $s_1\in V(J)$) or $s_2s_3,s_5s_6\in E(J)$ (if $s_2\in V(J)$). But
 $\mathcal{C}(s_1s_3)=\mathcal{C}(s_4s_6)$ and $\mathcal{C}(s_2s_3)=\mathcal{C}(s_5s_6)$, a contradiction.



Now suppose $V(Q)\not\subseteq V(C)$. Then $V(J)\subseteq V(S)\cup V(C)$. Up to permuting indices and the involution $s_i\leftrightarrow s_{i+3}$, the center is one of the following five types: \(
c_1s_1,s_1s_4, c_1s_2,s_1s_2, s_1s_5.
\)
In the first two types fewer than four compatible petal endpoints remain; in the last three types every disjoint petal pair repeats a color $\gamma_{ab}$ or $\delta_{ab}$, a contradiction.

\textit{Saturation.}
Let $H_1=G_3-\{s_4,s_5,s_6\}$ and $H_2=G_3-\{s_1,s_2,s_3\}$. Then $H_1$ and $H_2$ are isomorphic to the graph constructed consistently as $G_2$ with $q=2$. So the families $\Phi_{ij}$, $\Psi_{ij}$ and $\Omega_{ij}$ with $1\le i\not=j\le 3$ defined in the proof of Proposition~\ref{P520} give color-avoiding saturating copies except when, up to symmetry, $e=rs_i$ with $i\in\{1,2,3\}$ and $\alpha=\gamma_{jk}$, where $j,k\in \{1,2,3\}\setminus\{i\}$ and $j\not=k$. Every nonedge of $G_3$ lies in $H_1$ or $H_2$, and the remaining exceptional case is handled by
\(
	(s_1r\cup s_4s_5)\vee(c_1c_3),
\)
whose edges in $G_3$ avoid $\gamma_{23}$. \qed
\vspace{0.3em}

\newcommand{\casefigqtwo}{%
\begin{figure}[!ht]
\centering
\tikzset{
	casebox/.style={draw,dashed,rounded corners=5pt,line width=0.8pt},
	caseblock/.style={draw,rounded corners=4pt,minimum width=1.45cm,minimum height=0.62cm,font=\small,align=center,line width=0.9pt},
	casevertex/.style={circle,draw,fill=white,inner sep=0.8pt,minimum size=5.2mm,font=\scriptsize,line width=0.85pt},
	casejoin/.style={line width=1.2pt}
}
\definecolor{gammathirteencolor}{RGB}{90,110,145}
\begin{tikzpicture}[scale=0.76]
	\draw[casebox,fill=orange!5] (-5.75,-1.15) rectangle (-1.15,2.30);
	\node[font=\small] at (-3.45,2.62) {$S\cong K_6$};
	\coordinate (sone) at (-4.85,1.20);
	\coordinate (stwo) at (-3.45,1.60);
	\coordinate (sthree) at (-2.05,1.20);
	\coordinate (sfour) at (-4.85,-0.10);
	\coordinate (sfive) at (-3.45,-0.50);
	\coordinate (ssix) at (-2.05,-0.10);
	\draw[line width=0.55pt,draw=red!75!black] (sone)--(stwo);
	\draw[line width=0.55pt,draw=red!75!black] (sfour)--(sfive);
	\draw[line width=0.55pt,draw=red!75!black] (sthree)--(ssix);
	\draw[line width=0.55pt,draw=blue!75!black] (stwo)--(sthree);
	\draw[line width=0.55pt,draw=blue!75!black] (sfive)--(ssix);
	\draw[line width=0.55pt,draw=blue!75!black] (sone)--(sfour);
	\draw[line width=0.55pt,draw=gammathirteencolor] (sone)--(sthree);
	\draw[line width=0.55pt,draw=gammathirteencolor] (sfour)--(ssix);
	\draw[line width=0.55pt,draw=gammathirteencolor] (stwo)--(sfive);
	\draw[line width=0.55pt,draw=brown!85!black] (sone)--(sfive);
	\draw[line width=0.55pt,draw=teal!70!black] (sone)--(ssix);
	\draw[line width=0.55pt,draw=black!85] (stwo)--(sfour);
	\draw[line width=0.55pt,draw=lime!65!black] (stwo)--(ssix);
	\draw[line width=0.55pt,draw=pink!85!black] (sthree)--(sfour);
	\draw[line width=0.55pt,draw=brown!45!yellow!100] (sthree)--(sfive);
	\node[casevertex,fill=orange!8] at (sone) {$s_1$};
	\node[casevertex,fill=orange!8] at (stwo) {$s_2$};
	\node[casevertex,fill=orange!8] at (sthree) {$s_3$};
	\node[casevertex,fill=orange!8] at (sfour) {$s_4$};
	\node[casevertex,fill=orange!8] at (sfive) {$s_5$};
	\node[casevertex,fill=orange!8] at (ssix) {$s_6$};
	\node[caseblock,fill=yellow!16,minimum width=2.25cm] (Bt) at (-3.45,-1.7) {$B_t$};
	\node[font=\scriptsize] at (-3.45,-2.43) {added for $t\geq3$};

	\draw[casebox,fill=blue!5] (0.20,-1.35) rectangle (2.55,1.55);
	\node[font=\small] at (1.38,1.86) {$C\cong K_3$};
	\node[casevertex,fill=blue!7] (c1) at (1.38,0.82) {$c_1$};
	\node[casevertex,fill=blue!7] (c2) at (0.76,-0.28) {$c_2$};
	\node[casevertex,fill=blue!7] (c3) at (2.00,-0.28) {$c_3$};
	\draw[line width=0.75pt,draw=red!75!black] (c1)--(c2);
	\draw[line width=0.75pt,draw=gammathirteencolor] (c1)--(c3);
	\draw[line width=0.75pt,draw=blue!75!black] (c2)--(c3);

	\draw[casebox,fill=green!5] (4.50,-1.70) rectangle (7.50,1.70);
	\node[font=\small] at (6.00,2.38) {$R$};
	\node[font=\scriptsize] at (6.00,1.98) {independent roots};
	\node[casevertex,fill=green!7] (r1) at (5.20,0.45) {$r_1$};
	\node[casevertex,fill=green!7] (r2) at (6.80,0.45) {$r_2$};
	\node[casevertex,fill=green!7] (rd) at (5.20,-0.65) {$\cdots$};
	\node[casevertex,fill=green!7] (rn) at (6.80,-0.65) {$r$};

	\draw[casejoin,draw=brown!85!black] (-1.15,0.80)--(0.20,0.80);
	\draw[casejoin,draw=lime!65!black] (-1.15,0.25)--(0.20,0.25);
	\draw[casejoin,draw=pink!85!black] (-1.15,-0.30)--(0.20,-0.30);
	\node[font=\small,fill=white,inner sep=1pt] at (-0.48,1.12) {$S\vee C$};
	\draw[casejoin,draw=magenta!70!black] ([yshift=4pt]Bt.east)--(0.20,-0.7);
	\draw[casejoin,draw=cyan!65!black] ([yshift=-7pt]Bt.east)--(0.20,-1.1);
	\node[font=\scriptsize,fill=white,inner sep=1pt] at (-0.48,-1.93) {$B_t\vee C$};

	\draw[casejoin,draw=orange!80!black] (2.55,0.80)--(4.50,0.80);
	\draw[casejoin,draw=violet!80!black] (2.55,0.20)--(4.50,0.20);
	\draw[casejoin,draw=yellow!55!black] (2.55,-0.40)--(4.50,-0.40);
	\node[font=\small,fill=white,inner sep=1pt] at (3.53,1.12) {$C\vee R$};

\end{tikzpicture}
\caption{The graph $H_t$ obtained from $G_3$ by adjoining $B_t$ and the complete join $B_t\vee C$. There are no edges between any two of the sets $S$, $V(B_t)$, and $R$.}
\label{fig:Gt3}
\end{figure}
}

\vspace{0.5em}
\medskip\noindent\textbf{The $p=2$ and $q=1$ case.}

Let $S\cong K_4$ with vertex set $\{s_1,s_2,s_3,s_4\}$,
$C=\{c_1,c_2\}$ be an independent set, and  $R$ be the remaining vertices (assume $R\neq\emptyset$) which is also an independent set. Let  $G_4=C\vee (S\cup R)$.
Color the edges of $S$ by the three monochromatic matchings
\[
	M_1=\{s_1s_2,s_3s_4\},\qquad
	M_2=\{s_1s_3,s_2s_4\},\qquad
	M_3=\{s_1s_4,s_2s_3\},
\]
and identify the join colors by $\mathcal{C}(s_ic_j)=\mathcal{C}(s_{5-i}c_{3-j})$ for all $i\in\{1,2,3,4\}$ and $j\in\{1,2\}$. All other edges receive pairwise distinct fresh colors (see Figure 5).

\begin{figure}[!ht]
	\centering
	\tikzset{
		casebox/.style={draw,dashed,rounded corners=5pt,line width=0.8pt},
		caseblock/.style={draw,rounded corners=4pt,minimum width=1.45cm,minimum height=0.62cm,font=\small,align=center,line width=0.9pt},
		casevertex/.style={circle,draw,fill=white,inner sep=0.8pt,minimum size=5.2mm,font=\scriptsize,line width=0.85pt},
		casejoin/.style={line width=1.2pt}
	}
	\begin{tikzpicture}[scale=0.78]
		\draw[casebox,fill=orange!5] (-5.50,-0.72) rectangle (-1.40,2.06);
		\node[font=\small] at (-3.45,2.38) {$S\cong K_4$};
		\node[casevertex,fill=orange!8] (s1) at (-4.45,1.13) {$s_1$};
		\node[casevertex,fill=orange!8] (s2) at (-2.45,1.13) {$s_2$};
		\node[casevertex,fill=orange!8] (s3) at (-4.45,0.07) {$s_3$};
		\node[casevertex,fill=orange!8] (s4) at (-2.45,0.07) {$s_4$};
		\draw[line width=0.75pt,draw=red!75!black] (s1)--(s2);
		\draw[line width=0.75pt,draw=red!75!black] (s3)--(s4);
		\draw[line width=0.75pt,draw=blue!75!black] (s1)--(s3);
		\draw[line width=0.75pt,draw=blue!75!black] (s2)--(s4);
		\draw[line width=0.75pt,draw=green!50!black] (s1)--(s4);
		\draw[line width=0.75pt,draw=green!50!black] (s2)--(s3);

		\draw[casebox,fill=blue!5] (0.20,-1.35) rectangle (2.55,1.35);
		\node[font=\small] at (1.38,1.66) {$C=\{c_1,c_2\}$};
		\node[casevertex,fill=blue!7] (c1) at (0.85,0.18) {$c_1$};
		\node[casevertex,fill=blue!7] (c2) at (1.90,0.18) {$c_2$};
		\node[font=\scriptsize] at (1.38,-0.62) {$c_1c_2\notin E(G_4)$};
	
		\draw[casebox,fill=green!5] (4.50,-1.70) rectangle (7.50,1.70);
		\node[font=\small] at (6.00,2.38) {$R$};
		\node[font=\scriptsize] at (6.00,1.98) {independent roots};
		\node[casevertex,fill=green!7] (r1) at (5.20,0.45) {$r_1$};
		\node[casevertex,fill=green!7] (r2) at (6.80,0.45) {$r_2$};
		\node[casevertex,fill=green!7] (rd) at (5.20,-0.65) {$\cdots$};
		\node[casevertex,fill=green!7] (rn) at (6.80,-0.65) {$r$};
	
		\draw[casejoin,draw=brown!85!black] (-1.40,0.80)--(0.20,0.80);
		\draw[casejoin,draw=lime!65!black] (-1.40,0.25)--(0.20,0.25);
		\draw[casejoin,draw=pink!85!black] (-1.40,-0.30)--(0.20,-0.30);
		\node[font=\small,fill=white,inner sep=1pt] at (-0.60,1.12) {$S\vee C$};

		\draw[casejoin,draw=orange!80!black] (2.55,0.70)--(4.50,0.70);
		\draw[casejoin,draw=violet!80!black] (2.55,0.10)--(4.50,0.10);
		\draw[casejoin,draw=yellow!55!black] (2.55,-0.50)--(4.50,-0.50);
		\node[font=\small,fill=white,inner sep=1pt] at (3.53,1.02) {$C\vee R$};
	
	\end{tikzpicture}
	\caption{The  construction of $G_4$. Equal colors in $S$ indicate the three monochromatic matchings. There are no edges between $S$ and $R$.}
	\label{fig:case4}
\end{figure}

\begin{prop}\label{P522}
	 $G_4$  is $F_{2,2,1}$-rainbow saturated.
\end{prop}
\noindent {\bf Proof.}
\textit{Rainbow-freeness.} Suppose $J$ is a rainbow copy of  $F_{2,2,1}=2K_2 \vee K_1$ with central $Q=\{v\}$ in $G_4$. Then $v\notin R$. If $v\in C$, then both petals of $J$ lie in $S$ and hence form some $M_i$ ($1\le i\le 3$), a contradiction with $J$ being rainbow. If the center is $s_i$, then either $E(J)$ contains some
 $M_j$ for $1\le j\le 3$, or  the two petals of $J$ and their connecting edges to $s_i$ repeats a paired join color $\mathcal{C}(s_ac_b)=\mathcal{C}(s_{5-a}c_{3-b})$ for some $1\le a\le 4$ and $1\le b\le 2$, a contradiction.

\textit{Saturation.}
Fix a nonedge $e$ of $G_4$ and a color $\alpha$. We can assume  $e\in\{c_1c_2,r_1r_2,r_1s_1\}$, where $r_1,r_2\in R$. The matchings $M_1,M_2,M_3$ and the paired colors between $S$ and $C$ supply a rainbow copy of  $(2K_2 \vee K_1)$ in $G_4+_\alpha e$ if $e\in\{r_1r_2,r_1s_1\}$. So we just consider the case $e=c_1c_2$. Representative choices for $e=c_1c_2$ are $(c_2s_1\cup s_2s_3)\vee c_1$ when $\alpha$ is the color of $M_1$ or $M_2$, and $(c_2r_1\cup s_1s_2)\vee c_1$ when $\alpha$ is the color of $M_3$; an $S$--$C$ color is avoided by $(c_1c_2\cup s_as_b)\vee s_k$ with $k$ from the other paired class, and a $C$--$R$ color  is avoided by switching the center vertex.  In each case, we have  a rainbow copy of  $(2K_2 \vee K_1)$ in $G_4+_\alpha e$.\qed

\newcommand{\casefigqone}{%
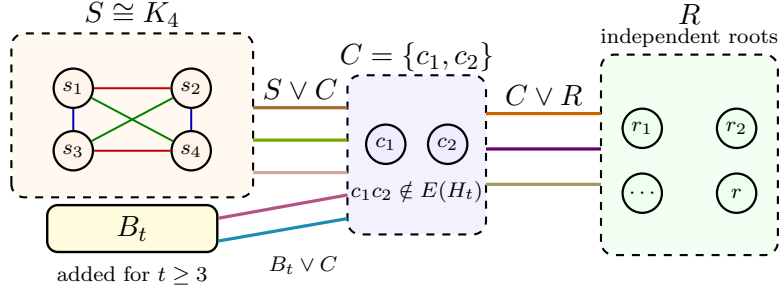
\begin{figure}[!ht]
\centering
\tikzset{
	casebox/.style={draw,dashed,rounded corners=5pt,line width=0.8pt},
	caseblock/.style={draw,rounded corners=4pt,minimum width=1.45cm,minimum height=0.62cm,font=\small,align=center,line width=0.9pt},
	casevertex/.style={circle,draw,fill=white,inner sep=0.8pt,minimum size=5.2mm,font=\scriptsize,line width=0.85pt},
	casejoin/.style={line width=1.2pt}
}
\begin{tikzpicture}[scale=0.78]
	\draw[casebox,fill=orange!5] (-5.50,-0.72) rectangle (-1.40,2.06);
	\node[font=\small] at (-3.45,2.38) {$S\cong K_4$};
	\node[casevertex,fill=orange!8] (s1) at (-4.45,1.13) {$s_1$};
	\node[casevertex,fill=orange!8] (s2) at (-2.45,1.13) {$s_2$};
	\node[casevertex,fill=orange!8] (s3) at (-4.45,0.07) {$s_3$};
	\node[casevertex,fill=orange!8] (s4) at (-2.45,0.07) {$s_4$};
	\draw[line width=0.75pt,draw=red!75!black] (s1)--(s2);
	\draw[line width=0.75pt,draw=red!75!black] (s3)--(s4);
	\draw[line width=0.75pt,draw=blue!75!black] (s1)--(s3);
	\draw[line width=0.75pt,draw=blue!75!black] (s2)--(s4);
	\draw[line width=0.75pt,draw=green!50!black] (s1)--(s4);
	\draw[line width=0.75pt,draw=green!50!black] (s2)--(s3);
	\node[caseblock,fill=yellow!16,minimum width=2.25cm] (Bt) at (-3.45,-1.25) {$B_t$};
	\node[font=\scriptsize] at (-3.45,-2.05) {added for $t\geq3$};

	\draw[casebox,fill=blue!5] (0.20,-1.35) rectangle (2.55,1.35);
	\node[font=\small] at (1.38,1.66) {$C=\{c_1,c_2\}$};
	\node[casevertex,fill=blue!7] (c1) at (0.85,0.18) {$c_1$};
	\node[casevertex,fill=blue!7] (c2) at (1.90,0.18) {$c_2$};
	\node[font=\scriptsize] at (1.38,-0.62) {$c_1c_2\notin E(H_t)$};

	\draw[casebox,fill=green!5] (4.50,-1.70) rectangle (7.50,1.70);
	\node[font=\small] at (6.00,2.38) {$R$};
	\node[font=\scriptsize] at (6.00,1.98) {independent roots};
	\node[casevertex,fill=green!7] (r1) at (5.20,0.45) {$r_1$};
	\node[casevertex,fill=green!7] (r2) at (6.80,0.45) {$r_2$};
	\node[casevertex,fill=green!7] (rd) at (5.20,-0.65) {$\cdots$};
	\node[casevertex,fill=green!7] (rn) at (6.80,-0.65) {$r$};

	\draw[casejoin,draw=brown!85!black] (-1.40,0.80)--(0.20,0.80);
	\draw[casejoin,draw=lime!65!black] (-1.40,0.25)--(0.20,0.25);
	\draw[casejoin,draw=pink!85!black] (-1.40,-0.30)--(0.20,-0.30);
	\node[font=\small,fill=white,inner sep=1pt] at (-0.60,1.12) {$S\vee C$};
	\draw[casejoin,draw=magenta!70!black] ([yshift=5pt]Bt.east)--(0.20,-0.68);
	\draw[casejoin,draw=cyan!65!black] ([yshift=-6pt]Bt.east)--(0.20,-1.08);
	\node[font=\scriptsize,fill=white,inner sep=1pt] at (-0.55,-1.87) {$B_t\vee C$};

	\draw[casejoin,draw=orange!80!black] (2.55,0.70)--(4.50,0.70);
	\draw[casejoin,draw=violet!80!black] (2.55,0.10)--(4.50,0.10);
	\draw[casejoin,draw=yellow!55!black] (2.55,-0.50)--(4.50,-0.50);
	\node[font=\small,fill=white,inner sep=1pt] at (3.53,1.02) {$C\vee R$};

\end{tikzpicture}
\caption{ $H_t$ obtained from $G_4$ by adjoining $B_t$ and the complete join $B_t\vee C$. There are no edges between any two of the sets $S$, $V(B_t)$, and $R$.}
\label{fig:Gt4}
\end{figure}
}

\subsection{Unified extension lemmas}

The three preceding constructions with $p=2$ are the base cases for two petals. The same bounded gadget extends each of them to an arbitrary number of petals.

Write $K_5^{\mathrm{rb}}$ for a rainbow-colored copy of $K_5$. Recall $A_2$ is an edge-colored copy of $K_4$ using two colors, where $\mathcal{C}(v_1v_2)=\mathcal{C}(v_1v_3)=\mathcal{C}(v_2v_4)=\mathcal{C}(v_3v_4)$ and $\mathcal{C}(v_2v_3)=\mathcal{C}(v_1v_4)$ (see Figure 1).
For $t\geq 3$, let $B_t$ denote the following bounded matching gadget; all its blocks use pairwise disjoint color sets:
\[
	B_t=\begin{cases}
		A_2, & t=3,\\
		\frac{t-2}{2}\,K_5^{\mathrm{rb}}, & t\geq 4 \text{ even},\\
		\frac{t-3}{2}\,K_5^{\mathrm{rb}}\cup A_2, & t\geq 5 \text{ odd}.
	\end{cases}
\]
In particular,
\[
	|V(B_t)|=\begin{cases}
		4, & t=3,\\
		\frac{5(t-2)}{2}, & t\geq 4 \text{ even},\\
		\frac{5t-7}{2}, & t\geq 5 \text{ odd}.
	\end{cases}
\]
\vspace{0.2em}

\begin{lemma}\label{LBT}
For every $t\geq 3$,  $B_t$ is $\mathcal{R}((t-2)K_2,2)$-saturated.
\end{lemma}
\noindent {\bf Proof.}
By Lemma~\ref{L516}, $A_2$ contains no rainbow $2K_2$ and, for every prescribed vertex of $A_2$ and color, $A_2$ contains an edge avoiding both. A rainbow $K_5$ contains no $3K_2$. Moreover, after deleting an arbitrary vertex of $K_5$ and excluding one prescribed color, the remaining $K_4$ still contains a rainbow $2K_2$, because at least one of its three perfect matchings avoids that color.

The graph $B_t$ is a disjoint union of these blocks, whose rainbow matching capacities add up to $t-2$.
This proves properties~(a) and~(b) of Definition~\ref{D510}. For~(c), every nonedge $e$ of $B_t$ joins two different blocks.
Then $B_t+_\alpha e$ contains a rainbow $(t-1)K_2$ formed by $e$ together with a maximum matching that avoids the endpoints of $e$ and the color $\alpha$ in each block. \qed
\vspace{0.3em}

\begin{lemma}\label{LBTextension}
Let $G_0=G_2,G_3,\text{and }G_4$ respectively when $q\ge 3$, $q=2$ and $q=1$. For $t\geq3$, let $H_t$ be an edge-colored graph with vertex set  $V(G_0)\cup V(B_t)$ and edge set
\[
	E(H_t)=E(G_0)\cup E(B_t)\cup \{xy: x\in V(B_t),y\in V(C)\},
\]
where the three edge sets $E(G_0)$, $E(B_t)$, and $E(V(B_t), C)$ use pairwise disjoint color sets, and the edges of $E(V(B_t), C)$ moreover receive pairwise distinct colors (see Figures~\ref{fig:Gt2}--\ref{fig:Gt4}). Then $H_t$ is $F_{t,2,q}$-rainbow saturated and has $(q+1)n+O_{t,q}(1)$ edges, where $q\ge 1$ and $n=|V(H_t)|$.
\end{lemma}

We will call $G_2,G_3,G_4$ the base graphs of $H_t$ and the colors used in $G_i$ ($2\le i\le 4$) the base colors.

\noindent{\bf Proof.} By Propositions~\ref{P520}--\ref{P522}, $G_2$ (resp. $G_3$ and $G_4$) is $(2K_2\vee K_q)$-rainbow saturated with $q\ge 3$ (resp. $(2K_2\vee K_2)$-rainbow saturated and $(2K_2\vee K_1)$-rainbow saturated).
They also have the following color-avoiding choices with the corresponding
value of $q$:
\begin{enumerate}
	\item every nonedge in the base graphs is saturated by a rainbow copy of $(2K_2\vee K_q)$ whose center is contained in $C$;
	\item for every $x\in V(S)\cup R$ and every  base color $\alpha$, there is a rainbow copy
	\[
		M\vee Q\cong K_2\vee K_q,~~\mbox{where}~~M\cong K_2,~~Q\cong K_q \mbox{~and~} V(Q)\subseteq V(C),
	\]
	such that $x\notin V(M)$, and each edge in $E(M \vee Q) \cup E(\{x\},V(Q))$ reserves a distinct color other than the color $\alpha$;
	\item for every $1\le j\le q+1$ and every $x\in V(S)\cup R$, there is $M\vee Q\cong K_2\vee K_q$ with $M\cong K_2$, $Q\cong K_q$ and $V(Q)=V(C)\setminus\{c_j\}$ and $x\notin V(M)$, and each edge in $E(M \vee Q) \cup E(\{x\},V(Q))$ reserves a distinct color.
\end{enumerate}

For $G_2$, these choices follow from the displayed families $\Phi_{ij}$, $\Psi_{ij}$ and $\Omega_{ij}$.
For $G_3$, use one of the two copies on $\{s_1,s_2,s_3\}\cup V(C)$ and $\{s_4,s_5,s_6\}\cup V(C)$; in (iii), one may choose proper $i$ and take
\(
s_is_j\vee(C[V(C)\setminus\{c_j\}]) \text{ or }  s_{i+3}s_{j+3}\vee(C[V(C)\setminus\{c_j\}]),
\)
according as $x\notin\{s_1,s_2,s_3\}$ or $x\notin\{s_4,s_5,s_6\}$, with $i\neq j$.
For $G_4$, the three monochromatic matchings and the paired $S$--$C$ colors give the same choices. These are finite checks from the displayed color identifications.

\textit{Rainbow-freeness.}
Suppose that $H_t$ contains a rainbow copy $J$ of $tK_2\vee K_q$, with center $Q$. It is not difficult to check that  $V(Q)\cap (V(S)\cup R)=\emptyset$. For any $u\in V(B_t)$, $d_{H_t}(u)\le q+4$
 when $t=3$ and $d_{H_t}(u)\le q+5$ when $t\geq4$, whereas a center vertex requires $q-1+2t$ neighbors. Hence $V(Q)\cap V(B_t)=\emptyset$. Thus $V(Q)=V(C)\setminus\{c_j\} $ for some $j$, say $V(Q)=V(C)\setminus\{c_1\} $.

By Lemma~\ref{LBT} and Definition~\ref{D510}(a), at most $t-2$ petals of $J$ lie completely in $B_t$. If no petal joins $B_t$ to $C$, at least two petals lie in $G_0$, a contradiction with $G_0$ being $(2K_2\vee K_q)$-rainbow saturated. Suppose there is a petal joining $B_t$ to $C$. Then  the unique possible cross-petal edge is $bc_1$ with $b\in V(B_t)$ which implies $V(J)\cap R=\emptyset$. Thus there is at least one petal of $J$, say $s_as_\ell$, lies completely in $S$.
When $q\geq3$, $s_as_\ell$ repeats the color of $c_ac_\ell$ either inside $Q$ or on an edge from $c_1$ to $Q$. When $q=2$, if the petal $s_as_\ell$ with the color $\gamma_{a\ell}$, then $s_as_\ell$ repeats a color used by $C$; if the petal $s_as_\ell$ with the color $\delta_{a(\ell-3)}$, then $s_as_\ell$  repeats one of its paired $S$--$C$ colors. When $q=1$, the cross-petal $bc_1$ is impossible because $c_1c_2\notin E(H_t)$. Thus $J$ has at most $t-1$ petals in every case, a contradiction.

\textit{Saturation.}
Fix a nonedge $e=uv$ and a color $\alpha$ to be avoided. If $u,v\in V(G_0)$, by property (i) above, there is a rainbow copy of $(2K_2\vee K_q)$ in $G_0+_{\alpha}e$ whose center is contained in $C$. By Lemma~\ref{LBT},  there are $t-2$ petals in $B_t$. We choose these $t-2$ petals to avoid $\alpha$ when $\alpha$ is used in $B_t$, and to avoid $b$ when $\alpha=\mathcal C(bc_j)$ where $b\in V(B_t)$ and $c_j\in V(C)$.

If $u,v\in V(B_t)$, by Lemma~\ref{LBT}, there are $t-1$ rainbow petals in $B_t+_{\alpha}e$ containing $e$. Complete the copy by the property (ii) when $\alpha$ is a base color, by the property (iii) when $\alpha=\mathcal C(bc_j)$ where $b\in V(B_t)$ and $c_j\in V(C)$, and by an arbitrary base petal and center when $\alpha$ is used in $B_t$.

Finally, suppose that  $u\in V(B_t)$ and $v\in V(S)\cup R$. Use $e$ as one petal and apply Definition~\ref{D510}(b) (via Lemma~\ref{LBT}) to obtain $t-2$ further rainbow petals in $B_t$, avoiding $u$ and $\alpha$ (when relevant). Property (ii) supplies the last base petal and center when $\alpha$ is a base color, and property (iii) does so when $\alpha=\mathcal{C}(b'c_j)$ for any $b'\in V(B_t)$ where the center $Q=C[V(C)\setminus\{c_j\}]$.
When $\alpha$ is used in $B_t$, take any choice in (iii).

The fresh join colors ensure that each selected copy in $H_t+_\alpha e$ is rainbow, contains $e$, and all its edges in $H_t$ avoid $\alpha$. \qed
\vspace{0.5em}

Figures~\ref{fig:case2}--\ref{fig:case4} show the three base graphs, and Figures~\ref{fig:Gt2}--\ref{fig:Gt4} show the corresponding graphs $H_t$.

\casefigqgeqthree
\casefigqtwo
\casefigqone
\FloatBarrier

\noindent {\bf Proof of Theorem~\ref{TFriend}.} The lower bound was proved at the beginning of this section. For the upper bound, Proposition~\ref{P519} handles $p\geq3$. When $p=2$, Propositions~\ref{P520}--\ref{P522} provide the three base graphs and Lemma~\ref{LBTextension} extends them to every $t\geq3$; the propositions themselves handle $t=2$. In every case the construction has $(p+q-1)n+O(1)$ edges. Therefore
\[
rsat(n,F_{t,p,q})=(p+q-1)n+O(1).\tag*{\qed}
\]

\section{Concluding remarks}

We have shown that isolated edges determine the order of magnitude of the rainbow saturation number: for fixed graphs without isolated vertices, $rsat(n,F)$ is bounded when $F$ contains an isolated edge, and is linear otherwise.

Theorem~\ref{TLower} gives a universal linear lower bound for graphs without an isolated edge, so a natural question is for which graphs the bound is sharp. This leads to the following problem.
\begin{prob}\label{prob:sharp}
Characterize the graphs $F$ without an isolated edge for which
\[
rsat(n,F)=\frac{\eta(F)}{2}n+o(n).
\]
\end{prob}

For generalized friendship graphs, Theorem~\ref{TFriend} determines the asymptotic behavior and leaves a more concrete exact problem.
\begin{prob}\label{prob:friendship}
For fixed integers $t\geq2$, $p\geq2$, and $q\geq1$, determine the exact value of $rsat(n,F_{t,p,q})$ for all sufficiently large $n$.
\end{prob}

Theorem~\ref{TFriend} assumes $p\geq2$ and therefore does not cover the case
\(
F_{t,1,q}=(tK_1)\vee K_q,
\)
where a clique $K_q$ is joined to an independent set of size $t$. When $q=1$, this graph is a star, which is discussed in Remark~\ref{R24}. This leads to the following question for $q\geq2$.
\begin{prob}\label{prob:friendship-p-one}
For fixed integers $t\geq2$ and $q\geq2$, determine
\(
rsat(n,F_{t,1,q}).
\)
\end{prob}

\section*{Acknowledgement}
The authors used AI
 for grammar and style checking.
It was not used to develop, verify, or modify any proof or mathematical argument in this paper.
After using this tool, the authors reviewed and edited the content as needed and take full responsibility for the content of the article.

The research of Lu is supported by the National Natural Science Foundation of China (No.~12171272). The research of Xu is supported by China Scholarship Council (No. 202606210184).


\begin{thebibliography}{99}

	\bibitem{BEH} N. Behague, T. Johnston, S. Letzter, N. Morrison and S. Ogden, The rainbow saturation number is linear, SIAM J. Discrete Math. 38 (2024), pp. 1239--1249.

	\bibitem{BR} C. Buchanan and M.P. Rombach, A lower bound on the saturation number and a strengthening for triangle-free graphs, Electron. J. Combin. 32 (2025), \#P3.14.

	\bibitem{Ca} A. Cameron and G.J. Puleo, A lower bound on the saturation number, and graphs for which it is sharp, Discrete Math. 345 (2022), 112867.

	\bibitem{CHA} D. Chakraborti, K. Hendrey, B. Lund and C. Tompkins, Rainbow saturation for complete graphs, SIAM J. Discrete Math. 38 (2024), pp. 1090--1112.

	\bibitem{survey} B.L. Currie, J.R. Faudree, R.J. Faudree and J.R. Schmitt, A survey of minimum saturated graphs, Electron. J. Combin. 18 (2011), \#DS19.

	\bibitem{EHM} P. Erd\H os, A. Hajnal and J.W. Moon, A problem in graph theory, Amer. Math. Monthly 71 (1964), pp. 1107--1110.

	\bibitem{Gir} A. Gir\~ao, D.C. Lewis and K. Popielarz, Rainbow saturation of graphs, J. Graph Theory 94 (2020), pp. 421--444.

	\bibitem{KT} L. K\'aszonyi and Z. Tuza, Saturated graphs with minimal number of edges, J. Graph Theory 10 (1986), pp. 203--210.

	\bibitem{Xu} Y. Xu, Z. He and M. Lu, The rainbow saturation number of cycles, arXiv:2501.06782.

\end{thebibliography}
\end{document}